\documentclass[10pt, reqno]{amsart}

\pdfoutput=1

\usepackage{amssymb, amsmath, amsthm}
\usepackage{enumitem}
\usepackage[all]{xy}
\usepackage[normalem]{ulem}
\usepackage{stmaryrd}
\usepackage{mathtools}
\usepackage{verbatim}
\usepackage{extarrows}
\usepackage{stackrel}
\usepackage{tikz}
\usepackage{mathrsfs}
\usepackage[french, english]{babel}

\usepackage{xcolor}

\colorlet{lightlightgray}{gray!20}
\colorlet{light1.5gray}{gray!35}

\usepackage[colorlinks=true, pdfstartview=FitV, linkcolor=blue, citecolor=blue, urlcolor=blue]{hyperref}



\newtheorem{theorem}{Theorem}[section]
\newtheorem*{theorem*}{Theorem}

\newtheorem*{maintheorem*}{Main Theorem}
\newtheorem{lemma}[theorem]{Lemma}
\newtheorem{corollary}[theorem]{Corollary}
\newtheorem{proposition}[theorem]{Proposition}

\newtheorem*{question*}{Question}

\theoremstyle{remark}
\newtheorem{remark}[theorem]{Remark}

\theoremstyle{definition}
\newtheorem{definition}[theorem]{Definition}
\newtheorem{claim}{Claim}

\setlist[1]{labelindent=\parindent, leftmargin=*}

\DeclareMathOperator{\GL}{GL}

\DeclareMathOperator{\Aut}{Aut}

\DeclareMathOperator{\Spec}{Spec}

\DeclareMathOperator{\Hom}{Hom}

\DeclareMathOperator{\rank}{rank}

\DeclareMathOperator{\Conv}{Conv}
\DeclareMathOperator{\Span}{Span}

\DeclareMathOperator{\inter}{int}

\DeclareMathOperator{\Mat}{Mat}
\DeclareMathOperator{\grp}{alg.grp}

\newcommand{\norm}[1]{\left\lVert#1\right\rVert}

\newcommand{\GG}{\mathbb{G}}
\newcommand{\RR}{\mathbb{R}}

\newcommand{\ZZ}{\mathbb{Z}}
\newcommand{\NN}{\mathbb{N}}
\renewcommand{\AA}{\mathbb{A}}

\newcommand{\kk}{\textbf{k}}

\newcommand{\name}[1]{#1}

\renewcommand{\phi}{\varphi}

\newcommand{\set}[2]{\left\{\,#1 \ | \ #2\,\right\}}
\newcommand{\Bigset}[2]{\left\{\,#1 \ \Big| \ #2\,\right\}}
\newcommand{\sprod}[2]{\langle #1, #2 \rangle}

\makeatletter
\ifcsname phantomsection\endcsname
\newcommand*{\qrr@gobblenexttocentry}[5]{}
\else
\newcommand*{\qrr@gobblenexttocentry}[4]{}
\fi
\newcommand*{\addsubsection}{%
	\addtocontents{toc}{\protect\qrr@gobblenexttocentry}%
	\subsection}
\makeatother

\title[On the Weights of Root Subgroups of Affine Toric Varieties]
{On the Weights of Root Subgroups of Affine Toric Varieties}
\author[Immanuel van Santen]
{Immanuel van Santen}

\thanks{}

\address{Departement Mathematik und Informatik, 
	Universit\"at Basel,\newline
	\indent Spiegelgasse 1, CH-4051 Basel, Switzerland}
\email{immanuel.van.santen@math.ch}

\begin{document}

\begin{abstract}
	Let $X$ be an affine toric variety and let $D(X)$ be the set of weights
	of all root subgroups. It is known that $D(X)$ together with its embedding into the
	character group determines $X$ as a toric variety.
	In this article we prove that $X$ is already determined by the abstract set $D(X)$
	together with some additional combinatorial data.
\end{abstract}

	
	\maketitle

	\setcounter{tocdepth}{1}
	\tableofcontents

	\section{Introduction}
	\label{sec.Introduction}

	Throughout this article we work over an algebraically closed
	field $\kk$ of characteristic zero. Let $T$ be an \emph{algebraic torus of dimension $n$}, 
	i.e.~$T$ is isomorphic to $\GG_m^n$,
	where $\GG_m$ denotes the multiplicative subgroup of $\kk$, and 
	let $X$ be an affine \emph{$T$-toric} variety, i.e.~$X$ is normal and
	endowed with a faithful algebraic $T$-action such that 
	there is an open dense $T$-orbit.  In this article, we study combinatorial
	data that determine $X$ as a toric variety.
	
	Let $x_0 \in T$ be an element in the open $T$-orbit
	and consider the open embedding $\eta_{x_0} \colon T \to T x_0 \subseteq X$.
	Denote by $\frak{X}(T)$ the \emph{character group of $T$}, i.e.~the group of all
	algebraic group homomorphisms $T \to \GG_m$, and define 
	\begin{eqnarray*}
		M_X &\coloneqq& 
		\Bigset{e \in \frak{X}(T)}{  \begin{array}{l}
										\textrm{$e$ extends to a morphism} \\ 
										\textrm{$X \to \AA^{1} \setminus \{0\}$ via $\eta_{x_0}$}
									\end{array} } \, , \\
		\Lambda(X) &\coloneqq& \Bigset{e \in \frak{X}(T)}{
			\begin{array}{l}
				\textrm{$e$ extends to a morphism} \\ 
				\textrm{$X \to \AA^{1}$ via $\eta_{x_0}$}
			\end{array} }
		 \, .
	\end{eqnarray*}
	The subgroup $M_X$ of $\frak{X}(T)$ and the subsemi-group $\Lambda(X)$ of $\frak{X}(T)$ 
	don't depend on the choice of $x_0$.
	It is well-known that $\Lambda(X)$
	determines $X$ as a toric variety. 

	In the study of the automorphism group $\Aut(X)$ of $X$ and related questions, 
	the root subgroups play an important role, 
	see e.g.~\cite{De1970Sous-groupes-algeb, Co1995The-homogeneous-co, Li2010Affine-Bbb-T-varie,
	ArZaKu2012Flag-varieties-tor, 
	LiReUr2023Characterization-o, ReSa2021Characterizing-smo, Ka2022Lines-in-affine-to}.
	Recall that a subgroup $U$ of $\Aut(X)$ is called a 
	\emph{root subgroup with respect to $T$}
	if there exists a $\GG_a$-action $\rho \colon \GG_a \times X \to X$ 
	of the additive subgroup $\GG_a$ of $\kk$ such that
	the image of the natural homomorphism $\GG_a \to \Aut(X)$ is equal to $U$ and
	if there exists $e \in \frak{X}(T)$ such that
	\[
		t \rho(s, t^{-1}x) = \rho(e(t)s, x) \qquad 
		\textrm{for all $s \in \GG_a$, $t \in T$, $x \in X$} \, .
	\]
	The character  $e$ is uniquely determined by $U$ and it is called
	the \emph{weight} of $U$. Let
	\[
		D(X) \coloneqq 
		\Bigset{e \in \frak{X}(T)}{  \begin{array}{l}
													\textrm{$e$ is the weight of a root subgroup} \\
											  		\textrm{in $\Aut(X)$ with respect to $T$}
										  		\end{array} } \, .
	\]
	It is natural to ask, to what extent $D(X)$ determines $X$. In fact, 
	the set $D(X)$ together with its embedding into the character group $\frak{X}(T)$ determines $X$, 
	see~\cite[Lemma~6.11]{LiReUr2023Characterization-o}. Moreover, 
	if $X$ is \emph{non-degenerate} 
	(i.e.~$M_X$ is trivial), then one has the following formula
	\begin{equation}
		\label{Eq.Formula_Weight_monoid}
		\Lambda(X) = \Conv(D(X)_{\infty}) \cap \frak{X}(T) \, ,
	\end{equation}
	where $D(X)_{\infty}$ denotes the so-called \emph{asymptotic cone} in 
	$\frak{X}(T) \otimes_{\ZZ} \RR$
	(see Sec.~\ref{subsec.Affine_toric_varieties} for the definition)
	and $\Conv$ stands for the convex hull, see~\cite[Corollary~8.4]{ReSa2021Characterizing-smo}.
	In this article, we show that in fact the (abstract) set $D(X)$
	together with some additional combinatorial data of $D(X)$ already determines $X$.

	\smallskip

	For the formulation of our main result we introduce some further notation. 
	The map that sends a root subgroup
	in $\Aut(X)$ with respect to $T$ onto its weight in $D(X)$ is a bijection. Moreover, 
	$D(X)$ is invariant under adding elements from $M_X$ and $D(X)$ can be naturally
	written as a disjoint union of subsets $S_{\rho_1}, \ldots, S_{\rho_r}$, 
	see Subsec.~\ref{subsec.Weights_root_subgroups}.
	In fact, the subsets $S_{\rho_1}, \ldots, S_{\rho_r}$ 
	of $D(X)$ can be characterized as those subsets of $D(X)$ that 
	are maximal with respect to the property that
	for any two elements of it the corresponding root subgroups commute, 
	see Lemma~\ref{Lem.Interpreting_Srho}.

		After choosing a basis of $\frak{X}(T)$ we may define
		for $(e_1, \ldots, e_n) \in \frak{X}(T)^n$ the absolute value of
		its determinant
		\[
			|\det(e_1, \ldots, e_{n})| \in \NN_0 \, .
		\]
	This quantity is independent of the choice of a basis of $\frak{X}(T)$.
	In fact, if $e_1, \ldots, e_{n}$ are linearly independent, then 
	$|\det(e_1, \ldots, e_{n})|$ is equal to the order of the finite subgroup
	$\bigcap_{i=1}^n \ker(e_i)$ of $T$, see 
	Lemma~\ref{Lem.Interpreting_Det}.									

	Furthermore, let $X'$ be an affine $T'$-toric variety for an $n$-dimensional algebraic
	torus $T'$. We may define the subgroup
	$M_{X'}$ and the subset $D(X')$ of $\frak{X}(T')$ analogously as for $X$, 
	and we have again a natural decomposition of $D(X')$ into disjoint subsets 
	$S_{\rho'_1}, \ldots, S_{\rho'_{r'}}$ for some integer $r'$. 
	
	\begin{maintheorem*}
			Assume there is a bijection
			$\Upsilon \colon D(X) \to D(X')$ such that
			\begin{enumerate}[label=\roman*)]
				\item \label{thm.properties_bijection_Theta_1} $r = r'$ and 
				there is a permutation $\xi$ of 
				$\{1, \ldots, r\}$ such that for all $i$
				\[\Upsilon(S_{\rho_i}) = S_{\rho'_{\xi(i)}} \, ; \]
				\item \label{thm.properties_bijection_Theta_2} for all $n$-tuples $(e_1, \ldots, e_n) \in D(X)^n$
						we have
						\[
							|\det(e_1, \ldots, e_n)| = 
							|\det(\Upsilon(e_1), \ldots, \Upsilon(e_n))| \, ;
						\]
				\item \label{thm.properties_bijection_Theta_3} 
						$\rank M_X = \rank M_{X'}$ and for all $e \in D(X)$ we have
						\[
							\Upsilon(e + M_X) = \Upsilon(e) + M_{X'} \, .
						\]
			\end{enumerate}
			Then $X$, $X'$ are isomorphic as toric varieties, 
			i.e.~there exists an isomorphism $\varphi \colon X \to X'$ of varieties
			and an isomorphism of algebraic groups $\varepsilon \colon T \to T'$ with
			\[
					\varphi(tx) = \varepsilon(t) \varphi(x) \quad 
					\textrm{for all $t \in T$, $x \in X$} \, .
			\]
	\end{maintheorem*}
	
	The proof of the main theorem is given at the end of 
	Subsec.~\ref{Subsec.The_strongly_convex_case}. 
	We use~property~\ref{thm.properties_bijection_Theta_3} in order to reduce
	to the case, when $X$ and $X'$ are both non-degenerate in 
	Subsec.~\ref{Subsec.Reduction_to_strongly_convex} and then tackle this case in 
	Subsec.~\ref{Subsec.The_strongly_convex_case}. 

	\begin{remark}
		If $\Upsilon \colon D(X) \to D(X')$ is a bijection that extends to a group isomorphism 
		$\Psi \colon M \to M'$, then using formula~\eqref{Eq.Formula_Weight_monoid} 
		(in fact a slightly more general form 
		of it) one can show that $\Psi$ maps $\Lambda(X)$ onto $\Lambda(X')$ and hence 
		$\Upsilon$ satisfies the 
		properties~\ref{thm.properties_bijection_Theta_1},~\ref{thm.properties_bijection_Theta_2}
		and~\ref{thm.properties_bijection_Theta_3} of the main theorem 
		(see Proposition~\ref{Prop.Reconstruction} for the details).
	\end{remark}
	
	\addsubsection*{Acknowledgements}
	The author would like to thank \name{Andriy Regeta}, \name{Jérémy Blanc} 
	and \name{Peter Feller} for many fruitful discussions.

	\section{Notation and setup concerning affine toric varieties}
	\label{sec.Notations}
	
	%
	%
	We refer the reader to~\cite{Fu1993Introduction-to-to}
	for the basic notions of affine toric varieties.
	Throughout this article, $T$ denotes an algebraic torus of dimension
	$n \geq 1$ and $X$ denotes an affine $T$-toric variety. Thus, we may consider $T$ as a subgroup of $\Aut(X)$. 
	We fix $x_0$ in the open $T$-orbit of $X$ and denote by
	$\eta_{x_0} \colon T \to X$, $t \mapsto t x_0$ the orbit map with respect to $x_0$.
	
	\subsection{General notions}
	
	Let $M = \frak{X}(T)$ 
	be the character group of $T$ and let $N \coloneqq \Hom_{\ZZ}(M, \ZZ)$ the group
	of one-parameter subgroups of $T$. Moreover, denote by $M_{\RR} \coloneqq M \otimes_{\ZZ} \RR$, 
	$N_{\RR} \coloneqq N \otimes_{\ZZ} \RR = \Hom_{\ZZ}(N, \RR)$ the extension by scalars with $\RR$ 
	and denote by
	\[
		M_{\RR} \times N_{\RR} \to \RR \, , \quad (u, v) \mapsto \sprod{u}{v} \coloneqq v(u)
	\]
	the natural perfect pairing between $M_{\RR}$ and $N_{\RR}$. 
	There is a unique
	strongly convex rational polyhedral cone $\sigma \subseteq N_{\RR}$ such that 
	its dual $\sigma^\vee \subseteq M_{\RR}$ satisfies 
	\[
		\kk[X] = \kk[\sigma^\vee \cap M] \subseteq \kk[M] = \kk[T] \, ,
	\]
	where the inclusion $\kk[X] \subseteq \kk[M]$ is induced by 
	the open embedding $\eta_{x_0} \colon T \to X$. Thus, 
	$\kk[X]$ is the subalgebra of those functions in $\kk[M]$ that extend to $X$
	via $\eta_{x_0}$ and therefore $\sigma^{\vee} \cap M = \Lambda(X)$.
	
	\subsection{The Weyl group}	
	\label{subsec.Weylgroup}
	The \emph{Weyl group} of $X$ is defined by
	\[
	W_X \coloneqq N_{\Aut(X)}(T)/T \, ,
	\]
	where $N_{\Aut(X)}(T)$ denotes the normalizer of $T$ in $\Aut(X)$.
	We consider the following action of the Weyl-group on $M$
	\[
		W_X \times M \to M \, , \quad 
		(\gamma, m) \mapsto \gamma \cdot m \coloneqq
		\left(t \mapsto m(\gamma^{-1} \circ t \circ \gamma) \right) \, .
	\]
	This $W_X$-action is linear on $M$, and it is faithful, since the centralizer
	of $T$ in $\Aut(X)$ is equal to $T$ (see e.g.~\cite[Lemma~2.10]{KrReSa2021Is-the-affine-spac}). 
	For every element in $W_X$ there exists a unique
	representative $\gamma \in N_{\Aut(X)}(T)$ with $\gamma(x_0) = x_0$.
	Since for all $t \in T$ and all $\gamma \in N_{\Aut(X)}(T)$ with $\gamma(x_0) = x_0$ we have
	\begin{equation}
		\label{Eq.W_X_action}
		\eta_{x_0}(\gamma^{-1} \circ t \circ \gamma) 
		= (\gamma^{-1} \circ t \circ \gamma)(x_0)
		= \gamma^{-1}(\eta_{x_0}(t)) \, ,
	\end{equation}
	it follows that the $W_X$-action on $M$ leaves the subsemi-group $\sigma^\vee \cap M$
	invariant and that this action is faithful. 
	Hence, via this action on $M$, the Weyl-group $W_X$ corresponds to those algebraic group 
	automorphisms in $\Aut_{\grp}(T)$ that extend to $X$ via $\eta_{x_0}$. 
	
	We consider now the case $X = T$, when $x_0$ denotes the neutral element in $T$.
	We choose an identification $T = \GG_m^n$, and thus we may identify 
	$W_X$ with $\Aut_{\grp}(T)$ as above. The latter group identifies with $\GL_n(\ZZ)$
	via the group isomorphism
	\[
		\begin{array}{rcl}
			\vartheta \colon \GL_n(\ZZ) \xrightarrow{B \mapsto (B^{-1})^t} 
			\GL_n(\ZZ) &\to& \Aut_{\grp}(T) \, , \\
			(a_{ij})_{ij} &\mapsto& 
			\left( (t_1, \ldots, t_n) \mapsto
			\left( \prod_{i=1}^n t_i^{a_{1i}}, \ldots, \prod_{i=1}^n t_i^{a_{ni}} \right) \right) \, ,
		\end{array}
	\]
	where $C^t$ denotes the transpose of a matrix $C$,
	and $M$ identifies with $\ZZ^n$ via the group isomorphism
	\[
		\phi \colon \ZZ^n \to M \, , \quad  (m_1, \ldots, m_n) \mapsto 
		\left( (t_1,\ldots, t_n) \mapsto t_1^{m_1} \cdots t_n^{m_n}  \right) \, .
	\]
	Due to~\eqref{Eq.W_X_action} and by the choice of $\vartheta$, the $W_X$-action on 
	$M$ satisfies
	\[
		\gamma \cdot m = \varphi(\vartheta^{-1}(\gamma) \cdot \phi^{-1}(m)) \, .
	\]

	%
	%
	
	\subsection{Decomposition of the affine toric variety}
	\label{subsec.Decomposition_affine_toric_variety}
	Let $N_0 = \Span_{\ZZ}(\sigma \cap N) \subseteq N$ and denote by $\sigma_0$ 
	the cone in $(N_0)_{\RR}$ that is equal to $\sigma$. Then
	$N_X \coloneqq N / N_0$ is a free abelian group.
	Dualizing the natural maps
	$N_0 \to N \to N_X$, $\sigma_0 \to \sigma \to \{0\}$ 
	gives 
	\[
		M_X \to M \stackrel{\tau}{\to} M_0 \, , \quad
		\sigma^\bot = \{0\}^\vee \to \sigma^\vee \to \sigma_0^\vee  \, ,
	\]	
	where $M_0 \coloneqq \Hom_{\ZZ}(N_0, \ZZ)$ and $M_X = \Hom_{\ZZ}(N_X, \ZZ)$.
	Taking the spectra gives thus a sequence of toric morphisms
	\[
	X_0 \coloneqq \Spec(\kk[\sigma^\vee_0 \cap M_0]) \to \Spec(\kk[\sigma^\vee \cap M])
	\to \Spec(\kk[M_X]) \eqqcolon T_{N_X} \, ,
	\]
	where $X_0$ is $T_0$-toric under the natural action coming from the $\rank(M_0)$-dimensional torus
	\[
		T_0 \coloneqq \Spec(\kk[M_0]) \, .
	\]
	Choosing a (non-natural) splitting of the initial exact sequence
	$N_0 \to N \to N_X$ allows us to write
	\[
	X = X_0 \times T_{N_X} \, , \
	M = M_0 \oplus M_X \quad
	\textrm{and} \quad
	\sigma^\vee = \sigma_0^\vee \oplus \sigma^\bot \, .
	\]
	The $T_0$-toric variety $X_0$ is then non-degenerate, i.e.~$\sigma_0^\vee$ is strongly convex
	or in other words $M_{X_0} = 0$.
	In particular, $W_{X_0}$ permutes the extremal rays of $\sigma_0^\vee$. 
	Since $\sigma_0^\vee$ generates $(M_0)_{\RR}$ (as $\sigma_0$ is strongly convex), 
	we get that $W_{X_0}$ is finite. Note that 
	\[
		M_X = M \cap \sigma^\bot \, 
	\]
	and that $\sigma^\bot$ is the largest linear subspace of $M_{\RR}$ that is contained in $\sigma^\vee$.
	Hence, $W_X$ leaves $M_X$ invariant and thus the projection
	$\tau \colon M \to M_0$ induces a surjection of the Weyl groups
	\[
		\pi_X \colon W_X \to  W_{X_0} \, .
	\]
	In geometric terms, $\pi_X$ maps a $\gamma \in \Aut_{\grp}(T)$
	that extends to an automorphism of $X$ via $\eta_{x_0}$ onto its restriction of the fibre of
	$\Spec(\kk[\sigma^\vee \cap M]) \to T_{N_X}$ over the neutral element of $T_{N_X}$.
	
	The elements from the kernel of $\pi_X$ may be identified with 
	those linear isomorphisms $M \to M$ that map $m + M_X$
	onto itself for all $m \in M$.	If one considers a splitting $M = M_0 \oplus M_X$
	and chooses a basis for $M_0$ and one for $M_X$, the kernel of $\pi_X$
	identifies under $\vartheta \colon \GL_n(\ZZ) \to \Aut_{\grp}(T)$
	(see Subsec.~\ref{subsec.Weylgroup}) with the following matrix group
	\begin{eqnarray*}
		\ker(\pi_X) &=& \Bigset{
			\begin{pmatrix}
				I & 0 \\
				B & A
			\end{pmatrix} \in \GL_n(\ZZ)
		}{A \in \GL_s(\ZZ) \, , \ B \in \Mat_{s \times (n-s)}(\ZZ)} \\
		&=& \Mat_{s \times (n-s)}(\ZZ) \rtimes \GL_s(\ZZ)	\, ,
	\end{eqnarray*}
	where $s = \rank M_X$ and $I \in \GL_{n-s}(\ZZ)$
	denotes the identity matrix.
	
		\subsection{Weights of root subgroups}
		\label{subsec.Weights_root_subgroups}
	%
	Let $\rho \subseteq \sigma$ be an extremal ray and 
	define
	\[
	S_\rho \coloneqq \set{m \in (\sigma_\rho)^\vee \cap M}{\sprod{m}{v_\rho} = -1} \subseteq M \, ,
	\]
	where $\sigma_\rho$ is the rational convex polyhedral cone in $M_{\RR}$ spanned by
	all the extremal rays of $\sigma$ except $\rho$ and $v_\rho$ is the unique primitive
	vector inside $\rho \cap N$. 
	If $\rho_1, \ldots, \rho_r$ are the extremal rays 
	of $\sigma$, then $S_{\rho_1}, \ldots, S_{\rho_r}$ are pairwise disjoint, non-empty 
	(see~\cite[Remark 2.5]{Li2010Affine-Bbb-T-varie}), the union is equal to 
	$D(X)$ and if we denote for an extremal ray $\rho$ of $\sigma$ and $e \in S_\rho$
	\[
	\partial_{\rho, e} \colon \kk[\sigma^\vee \cap M] \to \kk[\sigma^\vee \cap M] \, , \quad
	m \mapsto \sprod{m}{v_\rho} (e + m) \, ,
	\]
	we get a bijection:
	\[
	\begin{array}{rlcc}
	D(X) = \dot\bigcup_{i=1}^r & S_{\rho_i} & \to & 
	\left\{ \,
	\begin{array}{l}
	\textrm{locally nilpotent derivations of $\kk[\sigma^\vee \cap M]$ } \\
	\textrm{corresponding to root subgroups of $\Aut(X)$} \, 
	\end{array}
	\right\} \\
	& S_{\rho_i} \ni e & \mapsto & \partial_{\rho_i, e}
	\end{array}
	\]
	(see \cite[Lemma~2.6, Theorem~2.7]{Li2010Affine-Bbb-T-varie}). The right hand set 
	corresponds in a natural way bijectively to the set of root subgroups in $\Aut(X)$ with respect to $T$; see e.g.~\cite[\S1.5]{Fr2017Algebraic-theory-o}. 
	By~\cite[Lemma~10.3]{Ro2014Sums-and-commutato},
	the root subgroups corresponding to weights in $S_\rho$ do commute in $\Aut(X)$.
	
	Note that for an extremal ray $\rho \subseteq \sigma$ the set 
	$S_{\rho}$ is invariant under adding elements from $\sigma^\vee \cap \rho^\bot \cap M$
	and in particular also under adding elements from $M_X = M \cap \sigma^\bot$. The same
	holds thus for elements in $D(X)$ and therefore
	\[
		D(X) = \tau^{-1}(D(X_0)) \, .
	\]
	
	\subsection{The subsets $H_{E, \rho, d}, H^{\pm}_{E, \rho}$ and $B_{E, \rho, D}$ of $S_{\rho}$}
	We fix an extremal ray $\rho \subseteq \sigma$. Here, we introduce subsets of $S_\rho$ that
	turn out to be very useful in the study of the bijection $\Upsilon \colon D(X) \to D(X')$ from the
	introduction. Roughly speaking, with the aid of these sets we may introduce ``coordinates'' on $S_\rho$.
	
	In order to define them, we will fix once and for all a basis of $M$. In particular, we can compute
	the determinant of an $n$-tuple of elements in $M$ with respect to that basis
	(changing the basis of $M$ leads to a possible switch of the sign of the determinant).
	
	If $E = (e_1, \ldots, e_{n-1}) \in M^{n-1}$, then we denote
	\[
	\Span_{\RR}(E) \coloneqq \Span_{\RR} \{e_1, \ldots, e_{n-1} \}
	\]
	and for any integer $d \geq 0$ we write
	\[
	H_{E, \rho, d} \coloneqq \set{e \in S_{\rho}}{ \det(e_1, \ldots, e_{n-1}, e) =d} \, ,
	\]
	\[
	H_{E, \rho}^{+} \coloneqq \bigcup_{d \geq 0} H_{E, \rho, d}
	\, , \quad
	H_{E, \rho}^{-} \coloneqq \bigcup_{d \leq 0} H_{E, \rho, d} \, .
	\]
	
	Note that if  $\{e_1, \ldots,e_{n-1}\}$ is linearly independent in $M_\RR$, then 
	$H_{E, \rho, d}$ is the intersection of $S_{\rho}$ 
	with an affine hyperplane that is parallel to the hyperplane $\Span_{\RR}(E)$
	inside $M_{\RR}$.
	In the picture below, we illustrate the sets $\sigma^\vee \cap M$, $S_\rho$, $H_{E, \rho, d}$ and 
	$H_{E, \rho, 0}$. All sets in the picture should be thought to be intersected with the lattice $M$
	inside $M_\RR$.
	
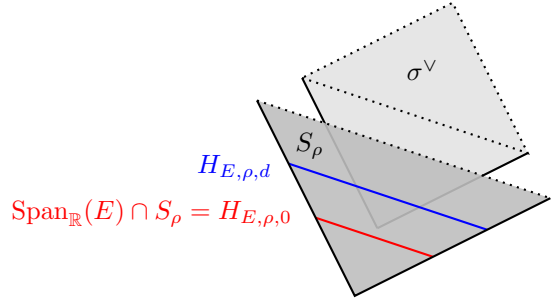
	\begin{center}
		\begin{tikzpicture}[scale=1]
		\fill[color=lightlightgray] (-1, 2) -- (0, 0) -- (2,  1) -- (1, 3);
		\draw[color=black, thick, dotted] (-1, 2) -- (2,  1) -- (1, 3) -- (-1, 2);
		\draw[color=black, thick] (0, 0) -- (-1, 2);
		\draw[color=black, thick] (0, 0) -- (2, 1);
		
		\fill[color=lightgray, opacity=0.9] (-1.6, 1.7) -- (-0.3, -0.9) -- (2.3, 0.4);
		
		\begin{scope}
			\path[clip] (-1.6, 1.7) -- (-0.3, -0.9) -- (2.3, 0.4)--cycle;
			\draw[color=blue, thick] (-1.6, 1) -- (2.3, -0.3);
			\draw[color=red, thick] (-1.6, 0.4) -- (2.3, -0.9);
		\end{scope}
		
		\draw[color=black, thick] (-1.6, 1.7) -- (-0.3, -0.9) -- (2.3, 0.4);
		\draw[color=black, thick, dotted] (-1.6, 1.7) -- (2.3, 0.4);	
		
		\draw[color = blue] (-1.9, 0.8) node {$H_{E, \rho, d}$};
		\draw[color = red] (-3, 0.2) node {$\Span_{\RR}(E) \cap S_\rho = H_{E, \rho, 0}$};
		
		\draw[] (0.6, 2.1) node {$\sigma^{\vee}$};
		\draw[] (-0.9, 1.1) node {$S_\rho$};
		\end{tikzpicture}
	\end{center}
	
	Moreover, if $E = (E_1, \ldots, E_{n-1})$, where  $E_i = (e_{i1}, \ldots, e_{i(n-1)}) \in D(X)^{n-1}$
	for $i=1, \ldots, n-1$ and if $D \geq 0$ is an integer, we denote
	\[
	B_{E, \rho, D} \coloneqq \Bigset{e \in S_\rho}{
							\begin{array}{l}
								\textrm{$|\det(e_{i1}, \ldots, e_{i(n-1)}, e)| \geq D$} \\
								\textrm{for all $i = 1, \ldots, n-1$}
							\end{array}}
	\subseteq S_\rho \, .
	\]	
	
	\section{General results on affine toric varieties}
	\label{subsec.Affine_toric_varieties}
	
	In the following section, we gather general results on affine toric varieties.
	We endow once and for all $M_{\RR}$ with a norm that we denote by $\norm{\cdot}$. 
	For a subset $S \subseteq M_{\RR}$, we denote by $S_{\infty} \subseteq M_{\RR}$
	its \emph{asymptotic cone}, i.e.
	\[
		S_\infty \coloneqq \left\{ x \in M_{\RR} \setminus \{0\} \ \Big| \  
		\begin{array}{l}
		\textrm{there exists a sequence $(x_i)_i$ in $S$} \\
		\textrm{with $\norm{x_i} \to \infty$ such that $x_i/\norm{x_i} \to x/\norm{x}$} 
		\end{array}
		\right\} \cup \{0\} \, .
	\]
	We refer the reader to~\cite[Chap. 2]{AuTe2003Asymptotic-cones-a} or \cite[\S2]{ReSa2021Characterizing-smo} for basic properties of asymptotic cones.
	Moreover, recall that the \emph{relative interior} of a subset $S \subseteq M_{\RR}$
	is the interior of $S$ inside $\Span_{\RR}(S)$. We will denote it by $\inter(S)$.
	
	\medskip
	
%
%
	
	We start with three easy observations:
	
	\begin{lemma}[{see \cite[Lemma~3.1]{ReSa2023Maximal-commutativ}}]
		\label{lem.span_facet}
		Let $\rho$ be an extremal ray of $\sigma$. Then the group generated by
		$\sigma^\vee \cap \rho^\bot \cap M$ is equal to $\rho^\bot \cap M$.
	\end{lemma}
	

	\begin{lemma}[{see \cite[Lemma~3.2]{ReSa2023Maximal-commutativ}}]
		\label{lem.Simultaneous_kernels}
		Let $e_1, \ldots, e_k$ in $M = \frak{X}(T)$, where $k \geq 1$. Then 
		\[
		K \coloneqq  \bigcap_{i=1}^k \ker(e_i) \simeq \GG_m^{n-r} \times F
		\quad \textrm{and} \quad
		\frak{X}(K) \simeq M / \Span_{\ZZ} \{e_1, \ldots, e_k \}
		\]
		for some finite commutative group $F$, where $r = \rank \Span_{\ZZ} \{e_1, \ldots, e_k \}$.
	\end{lemma}
	
	
	\begin{lemma}
		\label{lem.Finiteness_condition}
		Let $\rho \subseteq \sigma$ be an extremal ray and  
		let $C \subseteq M_{\RR}$ be a closed cone. If 
		$C \cap \sigma^\vee \cap \rho^\bot = \{0\}$, then $(u + C) \cap S_\rho$ is finite
		for all $u \in M_\RR$.
	\end{lemma}
	
	\begin{proof}[Proof of Lemma~\ref{lem.Finiteness_condition}]
		Assume that $(u + C) \cap S_\rho$ is non-finite. Hence, 
		there exists a sequence $(m_i)_{i \in \NN}$ inside $(u + C) \cap S_\rho$ such that
		$m_i \to \infty$ for $i \to \infty$ and such that 
		$\lim_{i \to \infty} \frac{m_i}{\norm{m_i}}$ exists in the unit sphere of $M_\RR$. Hence, 
		the asymptotic cone 
		$((u + C) \cap S_\rho)_{\infty}$ 
		contains non-zero elements. For fixed $e_0 \in S_\rho$ we
		get now a contradiction to the fact that $((u + C) \cap S_\rho)_{\infty}$
		is contained in
		\[
			(u + C)_{\infty} \cap (e_0 + \rho^\bot)_{\infty} 
			\cap ((\sigma_\rho)^\vee)_\infty
			= C \cap  \rho^{\bot} \cap (\sigma_\rho)^\vee
			= C \cap \rho^\bot \cap \sigma^\vee =  \{0\} \, . 
			\hfill \qedhere
		\]
	\end{proof}
	
	Roughly speaking, the next lemma says that any two elements from the topological
	boundary $\partial \sigma^\vee$ of $\sigma^\vee$ inside $M_\RR$ can be connected 
	in a ``nice way'' in case $\sigma^\vee$ is strongly convex and $n \geq 3$. 
	It will be used in the proof of Proposition~\ref{prop.psi}
	(which is the key ingredient in the proof of the main theorem).
	Recall that the topological boundary of $\sigma^\vee$ inside $M_{\RR}$  is given by
	\[
	\partial \sigma^\vee = 	\bigcup_{i=1}^r \sigma^\vee \cap \rho_i^\bot \, ,
	\]
	where $\rho_1, \ldots, \rho_r$ denote the extremal rays of $\sigma$.

	\begin{lemma}
		\label{lem.Schoenflies_application}
		Assume that $\sigma^\vee$ is strongly convex.
		Let $n \geq 3$, let $j_1, \ldots, j_s \in \{1, \ldots, r\}$ and 
		let $x, y \in \partial \sigma^\vee$ such that
		\begin{equation}
		\label{Eq.Assumption_star}
		\Span_{\RR}\left(\{x\} \cup \bigcup_{i=1}^s \sigma^\vee \cap \rho_{j_i}^\bot \right) = M_{\RR}
		= \Span_{\RR}\left(\{y\} \cup \bigcup_{i=1}^s \sigma^\vee \cap \rho_{j_i}^\bot \right) \, .
		\end{equation}
		Then there exist
		non-zero $z_0, \ldots, z_k \in \partial \sigma^\vee$ such that 
		\begin{enumerate}[label=\alph*)]
			\item \label{lem.Schoenflies_application_1} 
			$x = z_0$, $y = z_k$;
			\item \label{lem.Schoenflies_application_2} 
			for all $0 \leq l < k$ there exists a facet of $\sigma^\vee$
			that contains $z_l$ and $z_{l+1}$;
			\item \label{lem.Schoenflies_application_3} 
			$\Span_{\RR}(\{z_l\} \cup \bigcup_{i=1}^s \sigma^\vee \cap \rho_{j_i}^\bot) = M_{\RR}$
			for all $0 \leq l \leq k$.
		\end{enumerate}
	\end{lemma}
	
	\begin{proof}
		First we prove the following claim:
		\begin{claim}
			\label{Claim.Sphere}
			Let $C \subseteq M_{\RR}$ be a strongly convex rational polyhedral cone. 
			If $H \subseteq M_{\RR}$ is a hyperplane with $H \cap C = \{0\}$ and if
			$c \in \inter(C)$, then 
			\begin{enumerate}[label=\roman*)]
				\item \label{Claim.Sphere_i} $(c + H) \cap \partial C$ is homeomorphic to a sphere 
				of dimension $\dim \Span_{\RR}(C) -2$;
				\item \label{Claim.Sphere_ii} $\#(\RR_{> 0} c' \cap (c+H)) = 1$ 
				for all non-zero $c' \in C$,
			\end{enumerate}
			where $\partial C$ is the \emph{relative boundary} of $C$, i.e.~the boundary of
			$C$ inside $\Span_{\RR}(C)$.
		\end{claim}
		
		Indeed, for every non-zero $v \in H \cap \Span_{\RR}(C)$, the ray $c + \RR_{\geq 0}v$
		intersects $\partial C$ in exactly one point: as $C$ is convex, it intersects
		$\partial C$ in at most one point and if it doesn't intersect $\partial C$, then it would be contained
		in $C$ and thus we arrive at the following contradiction:
		\[
		\RR_{\geq 0} v
		= (c + \RR_{\geq 0}v)_{\infty} \subseteq ((c+ H) \cap C)_{\infty} \subseteq H \cap C = \{0\} \, .
		\]
		Note that the closed set $(c + H) \cap \partial C$ of $M_{\RR}$ 
		is bounded (and hence compact),
		since otherwise  $\{0\} \neq ((c + H) \cap \partial C)_{\infty}$, but this last asymptotic cone
		is contained in $H \cap C = \{ 0\}$. Hence, the map
		\[
		(c + H) \cap \partial C \to H \cap \Span_{\RR}(C) \cap S^{n-1} \, , \quad x \mapsto \frac{x-c}{\norm{x-c}}
		\]
		is a homeomorphism, where $S^{n-1} \subseteq M_{\RR}$ denotes the unit sphere with  centre at the origin. This proves Claim~\ref{Claim.Sphere}\ref{Claim.Sphere_i}. As $\RR_{\geq 0} c' \cap H = \{0\}$
		(as $C \cap H = \{0\}$), there exists a unique $t \in \RR$ with $tc' \in c + H$, i.e.~$c-tc' \in H$.
		If $t \leq 0$, then $c-tc' \in \inter(C)$, which contradicts $\inter(C) \cap H = \varnothing$.
		Hence, $t > 0$ and thus Claim~\ref{Claim.Sphere}\ref{Claim.Sphere_ii} follows.
		
		\medskip 
		
		Take $u \in \inter(\sigma^\vee)$. 
		As $\sigma^\vee$ is strongly convex, there exists a hyperplane
		$H$ in $M_{\RR}$ such that $\sigma^\vee \cap H = \{0\}$. 
		By Claim~\ref{Claim.Sphere}\ref{Claim.Sphere_i} 
		applied to $C = \sigma^\vee$ and $c = u$,
		the closed subset $(u + H) \cap \partial \sigma^\vee$ of $M_{\RR}$ is homeomorphic to 
		a sphere of dimension $n-2 \geq 1$. In particular $(u + H) \cap \partial \sigma^\vee$ is path connected,
		and thus we find non-zero $z_0, \ldots, z_k \in \partial \sigma^\vee$ that satisfy~\ref{lem.Schoenflies_application_1} and~\ref{lem.Schoenflies_application_2}.
		In case, not all $j_1, \ldots, j_s$ are equal, we have that 
		$\bigcup_{i=1}^s \sigma^\vee \cap \rho_{j_i}^\bot$ spans $M_{\RR}$
		(see Lemma~\ref{lem.span_facet})
		and thus~\ref{lem.Schoenflies_application_3} is automatically satisfied. Hence,
		we may assume that $\rho \coloneqq \rho_{j_1} = \ldots = \rho_{j_s}$.
		Then there exists 
		\[
			u_0 \in \inter(\sigma^\vee \cap \rho^\bot) \cap (u+H)
		\] 
		by Claim~\ref{Claim.Sphere}\ref{Claim.Sphere_ii} applied to $C = \sigma^\vee$, $c = u$
		and an element $c'$ from $\inter(\sigma^\vee \cap \rho^\bot)$.
		Hence, $u+ H= u_0 +H$ and by Claim~\ref{Claim.Sphere}\ref{Claim.Sphere_i}
		applied to $C = \sigma^\vee \cap \rho^\bot$ and $c = u_0$,
		the closed set $(u + H) \cap \partial (\sigma^\vee \cap \rho^\bot)$ is homeomorphic to a sphere
		of dimension $n-3$. By the generalized Schoenflies theorem 
		(see~ e.g.~\cite[Theorem~19.11]{Br1997Topology-and-geome}), 
		it follows that $(u + H) \cap ((\partial \sigma^\vee) \setminus \rho^\bot)$ is an $(n-2)$-dimensional 
		disc in the $(n-2)$-dimensional sphere
		$(u + H) \cap \partial \sigma^\vee$. The following picture illustrates the situation:
		
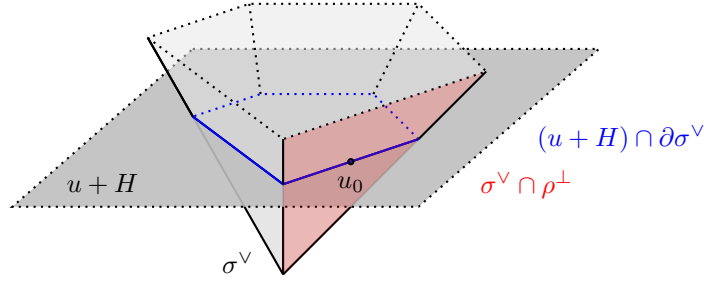
\begin{center}
			\begin{tikzpicture}[scale=0.6]
			\fill[color=lightlightgray] (-2, 3.5) -- (0, 0)-- (3, 3) -- (2, 4) -- (-0.5, 4);
			\draw[color=black, thick] (0, 0) -- (-2, 3.5);	
			\draw[color=black, thick] (0, 0) -- (0, 2);
			\draw[color=black, thick] (0, 0) -- (3, 3);
			\fill[color=lightgray, opacity=0.9]  (-2, 5) -- (7, 5) -- (3, 1.5) -- (-6, 1.5) -- (-2, 5);
			\draw[color=black, thick, dotted]  (-2, 5) -- (7, 5) -- (3, 1.5) -- (-6, 1.5) -- (-2, 5);
			\fill[color=lightlightgray, opacity=0.5] (-3, 5.25) -- (-2, 3.5) -- (0, 2) -- (3, 3) -- (4.5, 4.5) -- 
									(3, 6) -- (-0.75, 6) -- (-3, 5.25);
			\draw[color=black, thick] (-3, 5.25) -- (-2, 3.5) -- (0, 2) -- (3, 3);
			\draw[color=blue, thick] (-2, 3.5) -- (0, 2) -- (3, 3);
			\draw[color=blue, thick, dotted] (3, 3) -- (2, 4) -- (-0.5, 4) -- (-2, 3.5);
			\fill[color=red, opacity=0.2] (0, 0) -- (0, 3) -- (4.5, 4.5);
			\draw[color=black, thick] (-2, 3.5) -- (-3, 5.25);	
			\draw[color=black, thick] (0, 2) -- (0, 3);
			\draw[color=black, thick] (3, 3) -- (4.5, 4.5);
			\draw[color=black, thick, dotted] (2, 4) -- (3, 6);
			\draw[color=black, thick, dotted] (-0.5, 4) -- (-0.75, 6);
			\draw[color=black, thick, dotted] (-3, 5.25) -- (0, 3)-- (4.5, 4.5) -- (3, 6) -- (-0.75, 6) -- (-3, 5.25);
			\draw[] (-1, 0.25) node {$\sigma^\vee$};
			\draw[] (-4, 2) node {$u + H$};
			\draw[color=blue] (7.5, 3) node {$(u + H) \cap \partial \sigma^\vee$};
			\draw[color=red] (5.4, 2) node {$\sigma^\vee \cap \rho^\bot$};
			\draw[color=black, thick] (1.5, 2.5) circle (1.6pt);
			\draw[] (1.5, 2) node {$u_0$};
			\end{tikzpicture}
		\end{center} 
		
		Note that $x, y \in \partial \sigma^\vee \setminus \rho^\bot$ by~\eqref{Eq.Assumption_star}.
		Thus, we may find non-zero $z_0, \ldots, z_k \in (\partial \sigma^\vee) \setminus \rho^\bot$ 
		that satisfy~\ref{lem.Schoenflies_application_1} and~\ref{lem.Schoenflies_application_2}.
		As $z_i \not\in \rho^\bot$ for all $i=0, \ldots, k$,~\ref{lem.Schoenflies_application_3} is 
		satisfied as well.
	\end{proof}

	We finish this section by interpreting the assumptions~\ref{thm.properties_bijection_Theta_1}
	and~\ref{thm.properties_bijection_Theta_2} of the main theorem in group theoretic terms.

	\begin{lemma}
		\label{Lem.Interpreting_Srho}
		Each of the disjoint sets $S_{\rho_1}, \ldots, S_{\rho_r}$ of $D(X)$ is maximal
		with respect to the property that for any two elements the corresponding
		root subgroups commute.
	\end{lemma}

	\begin{proof}
		For $i=1,2$, let $\rho_i$ be an extremal ray in $\sigma$ and let $e_i \in S_{\rho_i}$.
		Then, by \cite[Proposition~1]{Ro2014Sums-and-commutato} we have the following:
		\[	
			\textrm{$\partial_{\rho_1, e_1}$, $\partial_{\rho_2, e_2}$ commute}
			\qquad \iff \qquad 
			\textrm{$\rho_1 = \rho_2$ \quad or \quad $e_1 \in \rho_2^{\perp} $ and $e_2 \in \rho_1^{\perp}$}
			\, .
		\]
		Assume $\partial_{\rho_1, e_1} \in S_{\rho_1}$ commutes with $\partial_{\rho_2, e}$
		for all $e \in S_{\rho_2}$. If $\rho_1 \neq \rho_2$, then there exists 
		$m_2 \in \sigma^\vee \cap \rho_2^{\bot} \cap M$ that doesn't belong to $\rho_1^\bot$.
		Fix $e_2 \in S_{\rho_2}$. Then $e_2 + m_2 \in S_{\rho_2}$ and $\sprod{e_2}{v_{\rho_1}} = 0$,
		and therefore $\sprod{e_2 + m_2}{v_{\rho_1}} \neq 0$, contradiction.
		This shows that $\rho_1 = \rho_2$ and thus the statement follows.
	\end{proof}

	\begin{lemma}
		\label{Lem.Interpreting_Det}
		If $e_1, \ldots, e_n \in M$ are linearly independent elements, then the integer
		$|\det(e_1, \ldots, e_n)|$ is equal to the order
		of the finite subgroup $\bigcap_{i=1}^n \ker(e_1)$ of $T$. 
	\end{lemma}

	\begin{proof}
		Assume that $\{e_1, \ldots, e_n \} \subseteq M$ is
		linearly independent. Moreover, denote by $K$ the intersection in $T$
		of all the kernels $\ker(e_1), \ldots, \ker(e_n)$. By Lemma~\ref{lem.Simultaneous_kernels}
		 \begin{equation}
		 		\label{Eq.M_mod_S}
		 		M/\Span_{\ZZ} \{e_1, \ldots, e_n\} \simeq \mathfrak{X}(K) \simeq K  \, . 
		 \end{equation}
		 Let $P \subseteq M_{\RR}$ be  the \emph{half-open parallelepiped spanned} by $e_1, \ldots, e_n$ 
		 in $M_{\RR}$, i.e.
		 \[
		 	P \coloneqq \set{\sum_{i = 1}^n \lambda_i e_i \in M_{\RR}}
		 		{\textrm{$0 \leq \lambda_i < 1$ for all $i = 1, \ldots, n$}} \subseteq M_{\RR} \, .
		 \]
		 By \cite[Lemma~10.3]{BeRo2007Computing-the-cont} we have
		 $|\det(e_1, \ldots, e_n)| = \#(M \cap P)$.
		 Moreover, $\#(M \cap P)$ is equal to $\#(M / \Span_{\ZZ} \{e_1, \ldots, e_n\})$.
		 This implies the statement.
	\end{proof}

	\section{Properties of the subsets 
	\texorpdfstring{$H_{E, \rho, d}, H^{\pm}_{E, \rho}$ and $B_{E, \rho, D}$ of $S_\rho$}
	{HErd, H+-Er, BErD of S_r}}
	Throughout this section, $\rho$ denotes an extremal ray of $\sigma$.
	We gather important properties of the subsets
	$H_{E, \rho, d}, H^{\pm}_{E, \rho}$ and $B_{E, \rho, D}$ of $S_\rho$.
	These properties will be used in the proof of the main theorem.
	
	\begin{lemma}
		\label{lem.Chara_all_H_e_rho_d_finite}
		Let $E = (e_1, \ldots, e_{n-1}) \in D(X)^{n-1}$ such that
		$H \coloneqq \Span_{\RR}(E)$ is a hyperplane in $M_{\RR}$.
		Then the following statements are equivalent
		\begin{enumerate}
			\item \label{lem.Chara_all_H_e_rho_d_finite_0} 
			for all $d \geq 0$ the set $H_{E, \rho, d} \cup H_{E, \rho, -d}$ is finite;			
			\item \label{lem.Chara_all_H_e_rho_d_finite_1} 
			for all $d \in \ZZ$ the set $H_{E, \rho, d}$ is finite;
			\item \label{lem.Chara_all_H_e_rho_d_finite_3}
			there exists $d \in \ZZ$ such that $H_{E, \rho, d}$ is finite and non-empty;
			\item \label{lem.Chara_all_H_e_rho_d_finite_4}
			we have $H \cap \rho^\bot \cap \sigma^\vee = \{0\}$.
		\end{enumerate}
		If one of the above equivalent statements is satisfied
		and $\sigma^\vee$ is strongly convex, then:
		\begin{enumerate}
			\setcounter{enumi}{4}
			\item \label{lem.Chara_all_H_e_rho_d_finite_2}
			$H_{E, \rho}^{+}$ or $H_{E, \rho}^{-}$ is finite.
		\end{enumerate}
	\end{lemma}
	
	\begin{proof}
		``\eqref{lem.Chara_all_H_e_rho_d_finite_0} $\Leftrightarrow$ \eqref{lem.Chara_all_H_e_rho_d_finite_1}''
		and
		``\eqref{lem.Chara_all_H_e_rho_d_finite_1} $\Rightarrow$ \eqref{lem.Chara_all_H_e_rho_d_finite_3}'':
		These statements are 
		clear, since $S_\rho = \bigcup_{d \in \ZZ} H_{E, \rho, d}$ 
		is non-empty.
		
		``\eqref{lem.Chara_all_H_e_rho_d_finite_3} $\Rightarrow$ \eqref{lem.Chara_all_H_e_rho_d_finite_4}'':
		Assume that
		$C \coloneqq H \cap \sigma^\vee \cap \rho^\bot \neq \{0\}$.
		Note that $C$ is a convex rational polyhedral cone in $M_{\RR}$. As $C$ is non-zero, by \cite[Lemma~2.5(1)]{ReSa2021Characterizing-smo} there exists a non-zero $u \in M \cap C$.
		Assume that $H_{E, \rho, d}$ is finite and non-empty. Then
		$e + t u \in H_{E, \rho, d}$ for all $t \in \ZZ_{\geq 0}$, $e \in H_{E, \rho, d}$ 
		and thus we get a contradiction to the finiteness of $H_{E, \rho, d}$.
		
		``\eqref{lem.Chara_all_H_e_rho_d_finite_4} $\Rightarrow$ \eqref{lem.Chara_all_H_e_rho_d_finite_1}'':
		Fix some $e_0 \in H_{E, \rho, d}$.  Then $H_{E, \rho, d} = (e_0 + H) \cap S_\rho$,
		since $H$ is a hyperplane in $M_{\RR}$. 
		Since $H \cap \sigma^\vee \cap \rho^\bot = \{0\}$,
		we get~\eqref{lem.Chara_all_H_e_rho_d_finite_1} from Lemma~\ref{lem.Finiteness_condition} applied
		to the closed cone $H$.
		
		``\eqref{lem.Chara_all_H_e_rho_d_finite_4} $\Rightarrow$ \eqref{lem.Chara_all_H_e_rho_d_finite_2}'':
		We assume from now on that $\sigma^\vee$ is strongly convex. Let 
		\[
		H^+ \coloneqq \set{w \in M_{\RR}}{\det(e_1, \ldots, e_{n-1}, w) \geq 0}
		\]
		and
		\[
		H^- \coloneqq \set{w \in M_{\RR}}{\det(e_1, \ldots, e_{n-1}, w) \leq 0} \, .
		\]
		As $\{e_1, \ldots, e_{n-1}\}$ is linearly independent in $M_\RR$,
		$H^+$ and $H^-$ are closed half spaces of $M_{\RR}$ such that 
		$H^+ \cap H^- = H$ and we have $H_{E, \rho}^{+} = H^+ \cap S_\rho$, 
		$H_{E, \rho}^{-} = H^- \cap S_\rho$.
		
		\begin{claim}
			\label{Claim.Dychotomy}
			$H^+ \cap \sigma^\vee \cap \rho^\bot = \{0\}$ or $H^- \cap \sigma^\vee \cap \rho^\bot = \{0\}$.
		\end{claim}
		Indeed, otherwise, there exist 
		non-zero $a^+ \in H^+ \cap \sigma^\vee \cap \rho^\bot$ and
		$a^- \in H^- \cap \sigma^\vee \cap \rho^\bot$. Hence, by~\eqref{lem.Chara_all_H_e_rho_d_finite_4}
		we have $a^+, a^- \not\in H$. Let $I \subseteq M_{\RR}$ be the affine line segment
		with end points $a^+$ and $a^-$. The convexity of $\sigma^\vee \cap \rho^\bot$ gives 
		$I \subseteq \sigma^\vee \cap \rho^\bot$.
		As $H \cap I$ is non-empty and as it is contained in $H \cap \sigma^\vee \cap \rho^\bot = \{0\}$, we
		get $I \cap H = \{0\}$. Hence, $\sigma^\vee \cap \rho^\bot$ contains the line 
		$L \coloneqq \set{\lambda u}{u \in I, \lambda \in \RR}$ in $M_\RR$.
		This contradicts the assumption that $\sigma^\vee$ is strongly convex and proves
		Claim~\ref{Claim.Dychotomy}.
	
		\medskip

		We get~\eqref{lem.Chara_all_H_e_rho_d_finite_2} from Claim~\ref{Claim.Dychotomy} and
		Lemma~\ref{lem.Finiteness_condition} applied to the closed cone $H^+$ or $H^-$ and $u = 0$.
	\end{proof}
	
	\begin{lemma}
		\label{lem.Chara_intersection_H_e_rho_d_i}
		Let  $E_i  \in D(X)^{n-1}$
		for $i = 1, \ldots, n-1$, such that  $H_i \coloneqq \Span_{\RR}(E_i)$ is a 
		hyperplane in $M_{\RR}$ and $H_i \cap \rho^\bot \cap \sigma^\vee = \{0\}$ for all
		$i$. Then the following statements are equivalent:
		\begin{enumerate}
			\item \label{lem.Chara_intersection_H_e_rho_d_i_1} 
			$\# \bigcap_{i=1}^{n-1} H_{E_i, \rho, d_i} \in \{0, 1\}$ for all $(d_1, \ldots, d_{n-1}) \in \ZZ^{n-1}$;
			\item \label{lem.Chara_intersection_H_e_rho_d_i_1.5}
			there exists $D \geq 0$ such that
			$\# \bigcap_{i=1}^{n-1} H_{E_i, \rho, d_i} \in \{0, 1\}$ for all $(d_1, \ldots, d_{n-1}) \in \ZZ^{n-1}$
			with $|d_1|, \ldots, |d_{n-1}| \geq D$;
			\item \label{lem.Chara_intersection_H_e_rho_d_i_2}
			$\bigcap_{i=1}^{n-1} H_i \cap \rho^\bot = \{0\}$.
		\end{enumerate}
	\end{lemma}
	
	\begin{proof}
		``\eqref{lem.Chara_intersection_H_e_rho_d_i_1} $\Rightarrow$ \eqref{lem.Chara_intersection_H_e_rho_d_i_1.5}'':
		This is clear.
		
		``\eqref{lem.Chara_intersection_H_e_rho_d_i_1.5} $\Rightarrow$ \eqref{lem.Chara_intersection_H_e_rho_d_i_2}'':
		If $n = 1$, then $\rho^\bot = \{0\}$ and hence we may assume that $n \geq 2$.
		Assume towards a contradiction that 
		the rational linear space $\bigcap_{i=1}^{n-1} H_i \cap \rho^\bot$ is non-zero. 
		Hence, there exists a non-zero $m_0 \in M \cap \bigcap_{i=1}^{n-1} H_i \cap \rho^\bot$.
		Let $e \in S_\rho$.
		
		Since $n \geq 2$, the set $S_{\rho}$ is infinite. 
		By Lemma~\ref{lem.Chara_all_H_e_rho_d_finite}, 
		for $1 \leq i \leq n-1$, 
		the sets $\bigcup_{d \colon |d| < D} H_{E_i, \rho, d}$ are finite and thus the subset
		\[
			H_D \coloneqq \bigcap_{i=1}^{n-1} 
			\bigcup_{d_i \colon |d_i| \geq D} H_{E_i, \rho, d_i}
		\]
		of $S_\rho$ is cofinite.
		Let $m_1 \in M$ be in the relative interior of $\sigma^\vee \cap \rho^\bot$. 
		Then $\sprod{m_1}{v_\mu} > 0$ for all extremal rays $\mu$ of $\sigma$ that are different from $\rho$.
		Hence, we may choose $t \in \ZZ_{\geq 0}$ such that 
		\begin{equation}
			\label{Eq.Inside_sigma_rho}
			\sprod{e + tm_1 + m_0}{v_{\mu}} \geq 0 \quad
			\textrm{for all extremal rays $\mu$ of $\sigma$ with $\mu \neq \rho$}
		\end{equation}
		and such that $e + tm_1 \in H_D$. 
		By~\eqref{lem.Chara_intersection_H_e_rho_d_i_1.5}, there exist unique 
		$d_1, \ldots, d_{n-1} \in \ZZ$ with $|d_1|, \ldots, |d_{n-1}| \geq D$
		such that 
		\begin{equation}
		\label{Eq.Singelton}
		\{e + tm_1\} = \bigcap_{i=1}^{n-1} H_{E_i, \rho, d_i} \, .
		\end{equation}
		Since $m_0 \in M \cap \bigcap_{i=1}^{n-1} H_i \cap \rho^\bot$ and using~\eqref{Eq.Inside_sigma_rho}, 
		it follows that 
		$e + tm_1 + m_0 \in \bigcap_{i=1}^{n-1} H_{E_i, \rho, d_i}$. Using that $m_0$ is non-zero,
		we arrive at a contradiction to~\eqref{Eq.Singelton}.

		``\eqref{lem.Chara_intersection_H_e_rho_d_i_2} $\Rightarrow$ \eqref{lem.Chara_intersection_H_e_rho_d_i_1}'':
		Assume that $e \in \bigcap_{i=1}^{n-1} H_{E_i, \rho, d_i}$ for some $(d_1, \ldots, d_{n-1}) \in \ZZ^{n-1}$.
		Since $H_i$ is a hyperplane, we have $H_{E_i, \rho, d_i} = (e + H_i) \cap S_{\rho}$ for all $i$. Hence,
		\[
		\{e\} \subseteq \bigcap_{i=1}^{n-1} H_{E_i, \rho, d_i} = \bigcap_{i=1}^{n-1} (e + H_i) \cap S_\rho \subseteq 
		\left( e + \bigcap_{i=1}^{n-1} H_i \right) \cap (e + \rho^\bot) 
		\xlongequal{\eqref{lem.Chara_intersection_H_e_rho_d_i_2}} \{e\} \, .
		\]
		This proves~\eqref{lem.Chara_intersection_H_e_rho_d_i_1}.
	\end{proof}

	In the next proposition we prove that there exit $E_1, \ldots, E_{n-1} \in D(X)^{n-1}$
	that satisfy the assumptions of Lemma~\ref{lem.Chara_intersection_H_e_rho_d_i}
	in case $\sigma^\vee$ is strongly convex.
	Before stating the proposition, we introduce a new term:
	
	\begin{definition}
		We call $E = (E_1, \ldots, E_{n-1}) \in (D(X)^{n-1})^{n-1}$
		an \emph{admissible system for $\rho$} if the linear spaces
		$H_i \coloneqq \Span_{\RR}(E_i)$, $i=1, \ldots, n-1$
		satisfy
		\begin{enumerate}
			\item \label{def.Admissible_System_1} 
			$H_i$ is a hyperplane in $M_{\RR}$ for all $i=1, \ldots, n-1$;
			\item \label{def.Admissible_System_2} $H_i \cap \sigma^\vee \cap \rho^\bot = \{0\}$ 
			for all $i=1, \ldots, n-1$;
			\item \label{def.Admissible_System_3} 
			$\bigcap_{i=1}^{n-1} H_i \cap \rho^\bot = \{0\}$.
		\end{enumerate}
	\end{definition}

	\begin{proposition}
		\label{prop.Existence_of_admissible_sys}
		Assume that $\sigma^\vee$ is strongly convex. Then there
		exists an admissible system for $\rho$.
	\end{proposition}

	%
	%
	%

	\begin{proof}[Proof of Proposition~\ref{prop.Existence_of_admissible_sys}]
		The dual cone of $\sigma^\vee \cap \rho^\bot$
		inside $\Hom_{\RR}(\rho^\bot, \RR)$ is of dimension $n-1$, as
		$\sigma^\vee$ is strongly convex. 
		Hence, we may choose $n-1$ linearly independent
		integral elements inside the relative interior of that dual cone
		and we denote by
		$A_1, \ldots, A_{n-1} \subseteq \rho^\bot$ the orthogonal rational hyperplanes to them
		inside $\rho^\bot$. Hence,
		we have
		\begin{enumerate}[label=\alph*)]
			\item \label{Eq.A_i_intersection_with-face} $A_i \cap \sigma^\vee = \{0\}$ for all
				$i=1, \ldots, n-1$;
			\item \label{Eq.Intersection_of_all_A_i} $\bigcap_{i=1}^{n-1} A_i =\{0\}$.
		\end{enumerate}
		Let $m_0 \in M$ be inside the relative interior of $\sigma^\vee \cap \rho^\bot$.
		For each $i=1, \ldots, n-1$ we choose linearly independent 
		$a_{i2}, \ldots, a_{i(n-1)} \in A_i \cap M$. By taking $t \in \ZZ_{\geq 0}$ large enough,
		we may assume that $tm_0 + a_{i2}, \ldots, tm_0 + a_{i(n-1)}$
		lie inside $M \cap \sigma^\vee \cap \rho^\bot$ for all $i=1, \ldots, n-1$. Now, 
		let $e_0 \in S_\rho$. Then the following elements from $S_\rho$ 
		\[
			e_{i1} \coloneqq e_0 + tm_0 \, , \ e_{i2} \coloneqq e_0 + t m_0 + a_{i2} \, , \ \ldots \, , \ 
			e_{i(n-1)} \coloneqq e_0 + t m_0 + a_{i(n-1)}
		\]
		are linearly independent in $M$.
		%
		Let $E_i \coloneqq (e_{i1}, \ldots, e_{i(n-1)}) \in D(X)^{n-1}$ for $i=1, \ldots, n-1$.
		In particular~\eqref{def.Admissible_System_1} is satisfied. Moreover, note that
		\[
			H_i \cap \rho^\bot = \Span_{\RR} \{ a_{i2}, \ldots, a_{i(n-1)}\} = A_i \, .
		\]
		Hence, \ref{Eq.A_i_intersection_with-face} and~\ref{Eq.Intersection_of_all_A_i}
		imply that~\eqref{def.Admissible_System_2} and~\eqref{def.Admissible_System_3} are satisfied.
	\end{proof}

	\section{Combinatorics of $D(X)$ and the proof of the main theorem}
	\label{sec.Combinatorics}
	
	Throughout this section, $X'$ denotes a $T'$-toric variety for some $n$-di\-men\-sional algebraic torus $T'$.
	We use the notation from the introduction 
	for $X$ and $X'$. Moreover, throughout this section,
	\[
		\Upsilon \colon D(X) \to D(X')
	\]
	denotes a bijection that satisfies 
	properties~\ref{thm.properties_bijection_Theta_1},~\ref{thm.properties_bijection_Theta_2}
	and~\ref{thm.properties_bijection_Theta_3} of 
	the main theorem. Thus after reordering $\rho'_1, \ldots, \rho'_{r'}$, $\Upsilon$ satisfies:
	\begin{enumerate}[label=\Roman*)]				
		\item \label{property.Upsilon_1} We have $r = r'$ and 
		$\Upsilon(S_{\rho_i}) = S_{\rho'_{i}}$ for all $i=1, \ldots, r$;
		\item \label{property.Upsilon_2} For all $n$-tuples $(e_1, \ldots, e_n) \in D(X)^n$
				we have
				\[
					|\det(e_1, \ldots, e_n)| = |\det(\Upsilon(e_1), \ldots, \Upsilon(e_n))| \, ;
				\]
		\item \label{property.Upsilon_3} We have
				$\rank M_X = \rank M_{X'}$ and
				for all $e \in D(X)$ we have
				\[
					\Upsilon(e + M_X) = \Upsilon(e) + M_{X'} \, .
				\]
	\end{enumerate}

	The goal of this section is to prove that under these assumptions the toric varieties 
	$X$ and $X'$ are isomorphic, i.e. the main theorem holds.
	
	\medskip
	
	We will use the following notation: If $E = (e_1, \ldots, e_{n-1}) \in D(X)^{n-1}$, then
	\[
		\Upsilon(E) \coloneqq (\Upsilon(e_1), \ldots, \Upsilon(e_{n-1})) \in D(X')^{n-1} \, .
	\]
	
	The following easy fact will be very useful in the sequel.
	\begin{lemma}
		\label{lem.H_E_rho_pmd}
		Let $E = (e_1, \ldots, e_{n-1}) \in D(X)^{n-1}$ and let $\rho$ be an extremal ray in $\sigma$.
		Assume that $\Upsilon(S_{\rho}) = S_{\rho'}$ for some extremal ray $\rho'$ of $\sigma'$.
		Then 
		\[
			\Upsilon(H_{E, \rho, d} \cup H_{E, \rho, -d}) = H_{\Upsilon(E), \rho', d} \cup H_{\Upsilon(E), \rho', -d}
			\quad \textrm{for all $d \in \ZZ$} \, .
		\]
	\end{lemma}
	
	\begin{proof}
		This is a direct consequence of~\ref{property.Upsilon_2}.
	\end{proof}
	
	\subsection{Reduction to the case when $\sigma^\vee$ is strongly convex}
	\label{Subsec.Reduction_to_strongly_convex}
	Throughout this subsection we use the notation from 
	Subsec.~\ref{subsec.Decomposition_affine_toric_variety}.
	Recall that we may write $X = X_0 \times \GG_m^s$, where $X_0$ is a  
	non-degenerate $T_0$-toric variety
	and consider the natural exact sequence $0 \to M_X \to M \stackrel{\tau}{\to} M_0 \to 0$
	of character groups. Note that
	\[
		s = \rank M_X \, .
	\]
	Note also that the extremal rays of the strongly convex cone 
	$\sigma \subseteq N$ that describes $X$ coincide  with the extremal rays of the strongly convex cone
	$\sigma_0 \subseteq N_0$ that describes $X_0$. Using that for all extremal rays $\rho \subseteq \sigma$
	the set  $S_\rho(X)$ is invariant under adding elements from $M_X$ yields thus
	\begin{equation}
		\label{Eq.S_rho(X)_S_rho(X_0)}
		S_\rho(X) = \tau^{-1}(S_{\rho}(X_0)) \, ,
	\end{equation}
	where $S_{\rho}(X)$ ($S_\rho(X_0)$) denotes the corresponding subset of $D(X)$
	($D(X_0)$), see Subsec.~\ref{subsec.Weights_root_subgroups}.
	Analogously we write $X' = X_0' \times \GG_m^{s'}$ and use analogous notations.
	
	\medskip
	
	The goal of this small subsection is to prove the following reduction result.
	Explicitly it will enable us to reduce to the case when $X$ and $X'$ are non-degenerate. 
	
	\begin{proposition}
		\label{prop.reduction_strongly_convex}
		We have $s = s'$ and
		there exists a bijection $\Upsilon_0 \colon D(X_0) \to D(X_0')$ that satisfies 
		properties~\ref{property.Upsilon_1}, \ref{property.Upsilon_2} and~\ref{property.Upsilon_3}.
	\end{proposition}

	For the proof we will need the following lemma:
	
	\begin{lemma}
		\label{lem.determinant_in_X_0}
		If $e_1^0, \ldots, e_{n-s}^0 \in D(X_0)$
		are linearly independent in $M_0$, then
		\[
			 |\det(e_1^0, \ldots, e_{n-s}^0)| 
			= \min \Bigset{|\det(e_1, \ldots, e_n)|}{
				\begin{array}{l}
					\textrm{$e_1, \ldots, e_n \in D(X)$ are} \\
					\textrm{linearly independent in $M$} \\
					\textrm{and $\{\tau(e_1), \ldots, \tau(e_n) \}$} \\
					\textrm{$=\{e^0_1, \ldots, e^0_{n-s} \}$}
				\end{array}
			} \, .
		\]
	\end{lemma}
	
	\begin{proof}
		We choose a splitting $M = M_X \oplus M_0$
		and identify $M_X = \ZZ^s$, $M_0 = \ZZ^{n-s}$. Then we get
		\[
		|\det(e_1^0, \ldots, e_{n-s}^0)| =
		\left|\det
		\begin{pmatrix}
		I_s & 0 \\
		\ast & e_1^0, \ldots, e_{n-s}^0
		\end{pmatrix} \right| \, ,
		\]
		where $I_s \in \GL_s(\ZZ)$ denotes the identity matrix. 
		This shows ``$\geq$‘‘ in the statement.
		
		Now, let $e_1, \ldots, e_n \in D(X)$ be linearly independent in $M$ such that
		\[
		\{\tau(e_1), \ldots, \tau(e_n) \} = \{e^0_1, \ldots, e^0_{n-s} \} \, .
		\] 
		Then there exist $a_1, \ldots, a_n \in M_X$ and a surjection
		$\eta \colon \{ 1, \ldots, n \} \to \{ 1, \ldots, n-s \}$ such that
		$e_i = a_i + e^0_{\eta(i)}$ for all $1 \leq i \leq n$.
		Applying a permutation $\varepsilon$ of $\{1, \ldots, n\}$ such that
		$\eta(\varepsilon(s+1)) =1, \ldots, \eta(\varepsilon(n)) = n-s$ gives us
		$b_1, \ldots, b_s \in M_X$ with
		\begin{eqnarray*}
			|\det(e_1, \ldots, e_n)| &=&
			\left|\det
			\begin{pmatrix}
				a_{\varepsilon(1)} & \ldots & a_{\varepsilon(s)} & a_{\varepsilon(s+1)} & \ldots & a_{\varepsilon(n)} \\
				e_{(\eta\circ \varepsilon)(1)}^0 & \ldots & e_{(\eta \circ \varepsilon)(s)}^0 &
				e_1^0 & \ldots & e_{n-s}^0
			\end{pmatrix} \right| \\
			&=& \left|\det
			\begin{pmatrix}
				b_1 & \ldots & b_s & a_{\varepsilon(s+1)} & \ldots & a_{\varepsilon(n)} \\
				0 & \ldots & 0 & e_1^0 & \ldots & e_{n-s}^0
			\end{pmatrix} \right| \\
			&\geq& |\det(e_1^0, \ldots, e_{n-s}^0)| \, ,
		\end{eqnarray*}
		where we used for the last inequality that $e_1, \ldots, e_n$ are linearly independent in $M$.
		Hence, ``$\leq$‘‘ holds as well in the statement.
	\end{proof}

	\begin{proof}[Proof of Proposition~\ref{prop.reduction_strongly_convex}]
		The first statement follows from property~\ref{property.Upsilon_3} for $\Upsilon$, as
		$s = \rank M_X = \rank M_{X'} = s'$.
		
		Again by property~\ref{property.Upsilon_3} for $\Upsilon$ and using that $D(X) = \tau^{-1}(D(X_0))$
		we get a well-defined bijection
		$\Upsilon_0 \colon D(X_0) \to D(X'_0)$ such that the following diagram commutes:
		\begin{equation}
			\begin{gathered}
			\label{Eq.Commutative_diagram_Upsilon}
			\xymatrix@=15pt{
				D(X) \ar[d]_-{\Upsilon} \ar[r]^-{\tau} & D(X_0) \ar[d]^-{\Upsilon_0} \\
				D(X') \ar[r]^-{\tau'} & D(X'_0) \, . \\
			}
			\end{gathered}
		\end{equation}
		Using~\eqref{Eq.S_rho(X)_S_rho(X_0)} and the commutativity of~\eqref{Eq.Commutative_diagram_Upsilon} 
		we get now property~\ref{property.Upsilon_1} for $\Upsilon_0$.
		
		Let $e^0_1, \ldots, e^0_{n-s} \in D(X_0)$.
		Again using the commutativity of~\eqref{Eq.Commutative_diagram_Upsilon} 
		gives us for all $(e_1, \ldots, e_n) \in D(X)^n$
		that
		\[
					\begin{array}{c}
							\textrm{$\{\tau(e_1), \ldots, \tau(e_n) \}$} \\
							\textrm{$=\{e^0_1, \ldots, e^0_{n-s} \}$}
					\end{array}
					\quad
					\iff
					\quad
					\begin{array}{c}
						 	\textrm{$\{\tau'(\Upsilon(e_1)), \ldots, \tau'(\Upsilon(e_n)) \}$} \\
							\textrm{$=\{\Upsilon_0(e^0_1), \ldots, \Upsilon_0(e^0_{n-s}) \}$}
					\end{array}					
		\]
		and by property~\ref{property.Upsilon_2} for $\Upsilon$ we have that
		$e_1, \ldots, e_n$ are linearly independent if and only if 
		$\Upsilon(e_1), \ldots, \Upsilon(e_n)$ are linearly independent.
		If $e^0_1, \ldots, e^0_{n-s}$ are linearly independent in $M_0$,
		then by Lemma~\ref{lem.determinant_in_X_0} 
		\[
			|\det(e_1^0, \ldots, e_{n-s}^0)| =
			|\det(\Upsilon_0(e_1^0), \ldots, \Upsilon_0(e_{n-s}^0))| 
		\]
		and in particular, $\Upsilon_0(e_1^0), \ldots, \Upsilon_0(e_{n-s}^0)$ are linearly
		independent in $M'_0$. Using the same argument for $\Upsilon^{-1}$
		gives us property~\ref{property.Upsilon_2} for $\Upsilon_0$.
		Property~\ref{property.Upsilon_3} is trivially satisfied for $\Upsilon_0$, as $M_{X_0}$ and $M_{X'_0}$
		are both trivial.
	\end{proof}
	
	\subsection{The case when $\sigma^\vee$ is strongly convex}
	\label{Subsec.The_strongly_convex_case}

	In order to prove the main theorem, by using
 	Proposition~\ref{prop.reduction_strongly_convex},
	we have only to consider the case when
	$X$ and $X'$ are non-degenerate, i.e.~$\sigma^\vee$ and $(\sigma')^\vee$ are 
	strongly convex. Note that $\sigma^\vee$ is strongly convex if and only if 
	$(\sigma')^\vee$ is strongly convex by property~\ref{property.Upsilon_3}.
	
	\medskip
	
	We start with a consequence of Lemma~\ref{lem.H_E_rho_pmd} that says that the
	subsets $H_{E, \rho, d}$ are preserved up to the sign of $d$ for large $d$, when 
	$\Span_{\RR}(E)$ intersects $\sigma^\vee \cap \rho^\bot$ only in the origin. This can be seen as
	an ``improvement of Lemma~\ref{lem.H_E_rho_pmd} up to a finite error''.
	
	\begin{corollary}
		\label{cor.H_E_rho_D_preserved}
		Assume that $\sigma^\vee$ is strongly convex.
		Let $\rho \subseteq \sigma$
		be an extremal ray and let $E = (e_1, \ldots, e_{n-1}) \in D(X)^{n-1}$
		such that $H \coloneqq \Span_{\RR}(E)$ 
		is a hyperplane in $M_{\RR}$ that
		satisfies 
		\begin{equation}
			\label{Eq.Assumption_for_the_hyperplane}
			H \cap \sigma^{\vee} \cap \rho^{\bot} = \{0\} \, .
		\end{equation}
		Assume that $\Upsilon(S_{\rho}) = S_{\rho'}$ for some extremal ray $\rho'$ of $\sigma'$.
		Then there exists $\varepsilon \in \{\pm1\}$ and $D \geq 0$ such that
		\[
			\Upsilon(H_{E, \rho, d}) = H_{\Upsilon(E), \rho', \varepsilon d} \quad
			\textrm{for all $d \in \ZZ$ with $|d| \geq D$} \, .
		\]
	\end{corollary}

	\begin{proof}
		Let $H' \coloneqq \Span_{\RR}( \Upsilon(E) )$. Since $H$ is a hyperplane,
		it follows from property~\ref{property.Upsilon_2} that $H'$ is a hyperplane as well.
		Moreover,
		\begin{eqnarray*}
			H \cap \sigma^{\vee} \cap \rho^{\bot} = \{0\} &\xLeftrightarrow{\textrm{Lem.~\ref{lem.Chara_all_H_e_rho_d_finite}}}&
			\textrm{$H_{E, \rho, d} \cup H_{E, \rho, -d}$ is finite $\forall \, d \geq 0$} \\
			&\xLeftrightarrow{\textrm{Lem.~\ref{lem.H_E_rho_pmd}}}& \textrm{$H_{\Upsilon(E), \rho', d} \cup H_{\Upsilon(E), \rho', -d}$ 
			is finite $\forall \, d \geq 0$} \\
			&\xLeftrightarrow{\textrm{Lem.~\ref{lem.Chara_all_H_e_rho_d_finite}}}&
			H' \cap (\sigma')^{\vee} \cap (\rho')^{\bot} = \{0\} \, .
		\end{eqnarray*}
	 	As $\sigma^\vee$ is strongly convex, by Lemma~\ref{lem.Chara_all_H_e_rho_d_finite} there exist $\delta, \delta' \in \{\pm1\}$ such that
	 	\[
	 		\# H_{E, \rho}^{\textrm{sign}(\delta)} < \infty \quad \textrm{and} \quad 
			\# H_{\Upsilon(E), \rho'}^{\textrm{sign}(\delta')} < \infty  \, . 
	 	\] 
	 	Hence, there exists $D \geq 0$ such that $H_{E, \rho, \delta d} = H_{\Upsilon(E), \rho, \delta' d} = \varnothing$
	 	for all $d \geq D$. Moreover, by using Lemma~\ref{lem.H_E_rho_pmd}, we get now
	 	for all $d \geq D$ that
	 	\begin{eqnarray*}
	 		\Upsilon(H_{E, \rho, -\delta d}) &=& \Upsilon(H_{E, \rho, -\delta d} \cup H_{E, \rho, \delta d}) \\
	 		&=& H_{\Upsilon(E), \rho', -\delta d} \cup H_{\Upsilon(E), \rho', \delta d} 
	 		= H_{\Upsilon(E), \rho', -\delta' d} \, .
	 	\end{eqnarray*}
	 	This gives the statement of the corollary by letting $\varepsilon \coloneqq \delta'/\delta \in \{\pm1\}$.
 	\end{proof}

 	\begin{proposition}
 	\label{prop.Theta_and_the_facets}
 	Assume that $\sigma^\vee$ is strongly convex.
 	Let $\rho$ be an extremal ray of $\sigma$ and let
 	$E = (E_1, \ldots, E_{n-1}) \in (D(X)^{n-1})^{n-1}$
 	be an admissible system for $\rho$. 
 	If $\Upsilon(S_{\rho}) = S_{\rho'}$ for some extremal ray $\rho'$ of $\sigma'$, then
 	there exists $D \geq 0$ such that the following holds:
 	\begin{enumerate}
 		\item \label{prop.Theta_and_the_facets_0}
 			$B_{E, \rho, D} + (M \cap \sigma^\vee \cap \rho^\bot) \subseteq B_{E, \rho, D}$;
  		\item \label{prop.Theta_and_the_facets_i} 
  			  for all $e, f \in B_{E, \rho, D}$ and all $m \in M \cap \sigma^\vee \cap \rho^\bot$ we have:
  			  \[
  					\Upsilon(e + m) - \Upsilon(e) = \Upsilon(f + m) - \Upsilon(f) \, ;
  			  \] 
 		\item \label{prop.Theta_and_the_facets_ii}
 			  for all $e \in B_{E, \rho, D}$ we have
 			  \[
 			  		\Upsilon(e + M \cap \sigma^\vee \cap \rho^\bot) = 
 			  		\Upsilon(e) + M' \cap (\sigma')^\vee \cap (\rho')^\bot \, .
 			  \]	  
 	\end{enumerate}
 	\end{proposition}
 	
 	For the proof of this proposition, the following remark will be essential:
 	
 	\begin{remark}
	\label{rem.Big_es}
	Let $\rho \subseteq \sigma$ be an extremal ray and let $E = (e_1, \ldots, e_{n-1}) \in D(X)^{n-1}$
	such that $H \coloneqq \Span_{\RR}(E)$ is a hyperplane in $M_{\RR}$
	and assume that there exist $\delta \in \{\pm1\}$ and $D \geq 0$ such that 
	$H_{E, \rho, -\delta d} = \varnothing$ for all $d \geq D$. Then we have for all $e \in S_\rho$
	and all $m \in M \cap \sigma^\vee \cap \rho^\bot$:
	\[
		|\det(e_{1}, \ldots, e_{n-1}, e)| \geq D \quad \implies \quad
		|\det(e_{1}, \ldots, e_{n-1}, e+m)| \geq D
	\]
	Indeed, we may assume without loss of generality that $\delta = +1$. Now,
	if $e \in S_\rho$ and $|\det(e_{1}, \ldots, e_{n-1}, e)| \geq D$, then 
	$\det(e_{1}, \ldots, e_{n-1}, e) \geq D$. If $\det(e_{1}, \ldots, e_{n-1}, m) < 0$,
	then $\det(e_{1}, \ldots, e_{n-1}, e+tm) \to -\infty$ for $t \to \infty$.
	However, $e + tm \in S_\rho$ for all $t \in \ZZ_{\geq 0}$ and thus we get a contradiction to
	$H_{E, \rho, -d} = \varnothing$ for all $d \geq D$.
	Hence, $\det(e_{1}, \ldots, e_{n-1}, m) \geq 0$ and thus $\det(e_{1}, \ldots, e_{n-1}, e + m) \geq D$.
 	\end{remark}
 
 	Moreover, the following description of $\sigma^\vee \cap \rho^\bot \cap M$ in terms of
 	$S_\rho$ will be used for the proof of Proposition~\ref{prop.Theta_and_the_facets}.
 	Recall that for a subset $S \subseteq M_{\RR}$, the set 
 	$S_\infty$ denotes the asymptotic cone in $M_{\RR}$, see 
	Sec.~\ref{subsec.Affine_toric_varieties}.
 	
 	\begin{lemma}
 		\label{lem.spcecial_cone_in_S_rho}
 		Let $\rho$ be an extremal ray in $\sigma$. Then, for all $e \in S_\rho$ we have
 		\[
 		M \cap \sigma^\vee \cap \rho^\bot = 
 		\set{m \in M}{\textrm{$e + tm \in S_\rho$ for all $t \in \ZZ_{\geq 0}$} } \, .
 		\]
 	\end{lemma}
 	
 	\begin{proof}[Proof of Lemma~\ref{lem.spcecial_cone_in_S_rho}]
 		The inclusion ``$\subseteq$'' is clear. Now, let $m \in M \setminus \{0\}$ such that 
 		$e + tm \in S_\rho$ for all $t \in \ZZ_{\geq 0}$. Hence,
 		\[
 		\frac{m}{\norm{m}} = \lim_{t \to \infty} \frac{e + tm}{\norm{e+tm}} \in (S_\rho)_\infty \subseteq
 		(\sigma_\rho^\vee)_{\infty} \cap (e + \rho^\bot)_{\infty} = 
 		\sigma_\rho^\vee \cap \rho^\bot = \sigma^\vee \cap \rho^\bot \, ,
 		\]
 		and thus $m \in M \cap \sigma^\vee \cap \rho^\bot$.
 	\end{proof}  

	\begin{proof}[Proof of Proposition~\ref{prop.Theta_and_the_facets}]
		Before starting with the proof of~\eqref{prop.Theta_and_the_facets_0}-\eqref{prop.Theta_and_the_facets_ii},
		we gather some preliminary facts. Within this proof, we will use the following notation:
		Denote $E_i = (e_{i1}, \ldots, e_{i(n-1)})$.
		For every $m \in M$, $m' \in M'$ and every $i \in \{1, \ldots, n-1\}$, we denote moreover
		\begin{equation}
			\label{Eq.definition}
			m_i \coloneqq \det(e_{i1}, \ldots, e_{i(n-1)}, m) 
			\quad \textrm{and} \quad
			m_i' \coloneqq \det(\Upsilon(e_{i1}), \ldots, \Upsilon(e_{i(n-1)}), m) \, .
		\end{equation}
	
		Let $H_i' \coloneqq \Span_{\RR} (\Upsilon(E_i))$. Due to property~\ref{property.Upsilon_2},
		Lemma~\ref{lem.Chara_all_H_e_rho_d_finite}, Lemma~\ref{lem.H_E_rho_pmd}, 
		Lemma~\ref{lem.Chara_intersection_H_e_rho_d_i} 
		and Corollary~\ref{cor.H_E_rho_D_preserved}, 
		\[
			\Upsilon(E) \coloneqq (\Upsilon(E_1), \ldots, \Upsilon(E_{n-1})) \in (D(X')^{n-1})^{n-1}
		\]
		is an admissible system for the extremal ray $\rho'$ of $\sigma'$
		by using that $E$ is an admissible system for the
		extremal ray $\rho$ of $\sigma$.
		Due to Corollary~\ref{cor.H_E_rho_D_preserved}, there exist $D \geq 0$ and 
		$\varepsilon_1, \ldots, \varepsilon_{n-1} \in \{\pm1\}$ with
		\begin{equation}
			\label{Eq.Preserving_H_Ei_rho_d}
			\Upsilon(H_{E_i, \rho, d}) = H_{\Upsilon(E_i), \rho', \varepsilon_i d} \quad
			\textrm{for all $d \in \ZZ$ with $|d| \geq D$} \, .
		\end{equation}
		Let $e \in B_{E, \rho, D}$. Then we have $|e_i| \geq D$ for all $i =1, \ldots, n-1$ and thus
		\[
			\{\Upsilon(e)\} \xlongequal{\textrm{Lem.~\ref{lem.Chara_intersection_H_e_rho_d_i}}}  
			\bigcap_{i=1}^{n-1} \Upsilon(H_{E_i, \rho, e_i}) \xlongequal{\eqref{Eq.Preserving_H_Ei_rho_d}}
			\bigcap_{i=1}^{n-1} H_{\Upsilon(E_i), \rho', \varepsilon_i e_i} \, .
		\]
		(see~\eqref{Eq.definition} for the definition of $e_i$). In other words, we have
		\begin{equation}
			\label{Eq.Image_of_point_under_Theta}
			\Upsilon(e)_i = \varepsilon_i e_i \quad \textrm{for all $i = 1, \ldots, n-1$} \, .
		\end{equation}
		
		\medskip
		
		\eqref{prop.Theta_and_the_facets_0}:
		By Lemma~\ref{lem.Chara_all_H_e_rho_d_finite}, there exist $\delta_1, \ldots, \delta_{n-1} \in \{\pm1\}$
		such that after possibly enlarging $D \geq 0$, we have
		$H_{E_i, \rho, -\delta_i d} = \varnothing$ for all $d \geq D$.
		Due to Remark~\ref{rem.Big_es}, we get thus the statement.
		
		\eqref{prop.Theta_and_the_facets_i}:
		Let $e \in B_{E, \rho, D}$ and $m \in M \cap \sigma^\vee \cap \rho^\bot$. We define $p_e \coloneqq \Upsilon(e+m)-\Upsilon(e)$.
		By~\eqref{prop.Theta_and_the_facets_0} we have $e+m \in B_{E, \rho, D}$. Thus we get
		for all $i = 1, \ldots, n-1$:
		\[
			\varepsilon_i(e_i + m_i) 
			\xlongequal{\eqref{Eq.Image_of_point_under_Theta}} 
			\Upsilon(e+m)_i
			= (\Upsilon(e) + p_e)_i = \Upsilon(e)_i + (p_e)_i 
			\xlongequal{\eqref{Eq.Image_of_point_under_Theta}} \varepsilon_i e_i + (p_e)_i \, .
		\]
		Hence, $\varepsilon_i m_i = (p_e)_i$ for all $i$. If $f \in B_{E, \rho, D}$ is another element, then 
		we get for all $i =1, \ldots, n-1$
		\[
			(\Upsilon(e+m)-\Upsilon(e))_i = (p_e)_i = \varepsilon_i m_i = (p_f)_i = (\Upsilon(f+m)-\Upsilon(f))_i \, .
		\]
		This gives for all $i=1, \ldots, n-1$
		\[
			\det(\Upsilon(e_{i1}), \ldots, \Upsilon(e_{i(n-1)}), \Upsilon(e+m)-\Upsilon(e)- \Upsilon(f+m) +\Upsilon(f)) = 0 \, .
		\]
		Since the elements $\Upsilon(e+m), \Upsilon(e), \Upsilon(f+m), \Upsilon(f)$ are lying in $S_{\rho'}$,
		it follows that 
		$\Upsilon(e+m)-\Upsilon(e) - \Upsilon(f+m) + \Upsilon(f) \in (\rho')^\bot$.
		In summary, we achieve
		\[
			\Upsilon(e+m)-\Upsilon(e)- \Upsilon(f+m) +\Upsilon(f) \in \bigcap_{i=1}^{n-1} H_i' \cap (\rho')^\bot =  \{0\} \, .
		\]
		Hence, $\Upsilon(e+m)-\Upsilon(e) = \Upsilon(f+m)-\Upsilon(f)$.
		
		\medskip
		
		\eqref{prop.Theta_and_the_facets_ii}: Let $e \in B_{E, \rho, D}$ and let 
		$m \in M \cap \sigma^\vee \cap \rho^\bot$. Hence, we have
		$e + tm \in S_\rho$ for all $t \in \ZZ_{\geq 0}$ and by~\eqref{prop.Theta_and_the_facets_0} we have
		$e + tm \in B_{E, \rho, D}$ for all $t \in \ZZ_{\geq 0}$. This gives
		\[
			\Upsilon(e+tm) - \Upsilon(e) = \sum_{i=0}^{t-1} \Upsilon(e+(t-i)m)-\Upsilon(e+(t-i-1)m)
			\stackrel{\eqref{prop.Theta_and_the_facets_i}}{=} t(\Upsilon(e+m)-\Upsilon(e))
		\]
		and thus 
		$\Upsilon(e) + t(\Upsilon(e+m)-\Upsilon(e)) = \Upsilon(e+tm) \in S_{\rho'}$
		for all $t \in \ZZ_{\geq 0}$. Lemma~\ref{lem.spcecial_cone_in_S_rho} gives us
		$\Upsilon(e+m)-\Upsilon(e) \in M' \cap (\sigma')^\vee \cap (\rho')^\bot$.
		This implies thus 
		\begin{equation}
			\label{Eq.A_Cone}
			\Upsilon(e + M \cap \sigma^\vee \cap \rho^\bot) \subseteq
			\Upsilon(e) + M' \cap (\sigma')^\vee \cap (\rho')^\bot \, .
		\end{equation}
		By using that $\Upsilon(E)$ is an admissible system for the extremal ray $\rho'$ of $\sigma'$
		and by possibly enlarging $D$,
		we may apply the above argument to $\Upsilon^{-1}$, and thus we get
		\begin{equation}
			\label{Eq.B_Cone}
			\Upsilon^{-1}(\Upsilon(e) + M' \cap (\sigma')^\vee \cap (\rho')^\bot) \subseteq
			e + M \cap \sigma^\vee \cap \rho^\bot \, .
		\end{equation}
		The inclusions~\eqref{Eq.A_Cone} and~\eqref{Eq.B_Cone} yield now~\eqref{prop.Theta_and_the_facets_ii}.
	\end{proof}

	Using Proposition~\ref{prop.Theta_and_the_facets}, we will prove in case
	$\Upsilon(S_\rho) = S_{\rho'}$ that there exists a unique $\RR$-linear map
	$\psi_\rho \colon M_{\RR} \to M'_{\RR}$ that extends $\Upsilon |_{S_\rho} \colon S_\rho \to S_{\rho'}$
	up to finitely many elements of $S_\rho$:

	\begin{corollary}
		\label{cor.facets}
		Assume that $\sigma^\vee$ is strongly convex and let $\rho$ be an extremal ray of $\sigma$.
		If $\Upsilon(S_{\rho}) = S_{\rho'}$ for some extremal ray $\rho'$ of $\sigma'$, then
		there exists a unique $\RR$-linear isomorphism  
		$\psi_{\rho} \colon M_{\RR} \to M'_{\RR}$ such that
		\[
			\psi_{\rho}(e) = \Upsilon(e) \quad  \textrm{for almost all $e \in S_\rho$
			(for all $e \in S_\rho$ if $n = 1$)} \, .
		\]
		Moreover, $\psi_\rho$ satisfies:
		\begin{enumerate}
			\item \label{cor.facets_1} $\psi_{\rho}(M) = M'$;
			\item \label{cor.facets_2} $\psi_{\rho}(\rho^\bot) = (\rho')^\bot$; 
			\item \label{cor.facets_3} $\psi_{\rho}(\sigma^\vee \cap \rho^\bot) = (\sigma')^\vee \cap (\rho')^\bot$.
		\end{enumerate}
	\end{corollary}

	\begin{proof}
		We may assume that $n \geq 2$, since otherwise $S_\rho$ and $S_{\rho'}$
		are singletons and thus the statement is clear.
		
		\textbf{Uniqueness:} Assume that $\psi_\rho, \bar{\psi}_\rho \colon M_{\RR} \to M'_{\RR}$ are $\RR$-linear
		isomorphisms such that there exists a finite set $A \subseteq S_\rho$ with
		$\psi_\rho(e) = \Upsilon(e) = \bar{\psi}_\rho(e)$ for all $e \in S_\rho \setminus A$.
		Let $m_0 \in M$ be in the relative interior of $\sigma^\vee \cap \rho^\bot$ and let $e_0 \in S_\rho$.
		As $A$ is finite, there exists $t \in \ZZ_{\geq 0}$ 
		such that 
		\[
			\sprod{e_0 + t m_0}{v_\mu} = \underbrace{\sprod{e_0}{v_\mu}}_{\geq 0} + t 
			\underbrace{\sprod{m_0}{v_\mu}}_{>0} > \sprod{a}{v_{\mu}} \quad
			\textrm{for all $a \in A$}
		\]
		and for all extremal rays $\mu$ of $\sigma$
		that are different from $\rho$. 
		Hence, 
		\[
			\sprod{e_0 + tm_0 + m}{v_\mu} >  \sprod{a}{v_{\mu}} \quad \textrm{for all $\mu \neq \rho$, 
			$m \in \sigma^\vee \cap \rho^\bot \cap M$ and $a \in A$} \, .
		\]
		This gives 
		\[
			(e_0 + tm_0) + (\sigma^\vee \cap \rho^\bot \cap M) \subseteq S_\rho \setminus A \, ,
		\]
		since $n \geq 2$ (note that there exist extremal rays $\mu \neq \rho$ in $\sigma$).
		Since  $\sigma^\vee \cap \rho^\bot \cap M$ spans $M \cap \rho^\bot$ as a group
		(see Lemma~\ref{lem.span_facet}),
		we may find $m_1, \ldots, m_{n-1} \in \sigma^\vee \cap \rho^\bot \cap M$
		that are linearly independent in $M$.
		Hence, the elements 
		\[
			e_0 + tm_0 \, , \ e_0 + tm_0 + m_1 \, , \ \ldots \, ,  e_0 + tm_0 +m_{n-1} \in S_\rho \setminus A  
		\]
		form a basis of $M_\RR$. Since $\psi_\rho(e) = \bar{\psi}_\rho(e)$ for all 
		$e \in S_\rho \setminus A$, we get $\psi_\rho = \bar{\psi}_\rho$.
		
		\textbf{Existence:} 
		Let $E = (E_1, \ldots, E_{n-1})$ be an admissible system for $\rho$ (see Proposition~\ref{prop.Existence_of_admissible_sys}) and
		let $D \geq 0$ that 
		satisfies~\eqref{prop.Theta_and_the_facets_0}-\eqref{prop.Theta_and_the_facets_ii} from Proposition~\ref{prop.Theta_and_the_facets}.
		Moreover, we fix $e_0 \in B_{E, \rho, D}$. According to~Proposition~\ref{prop.Theta_and_the_facets}\eqref{prop.Theta_and_the_facets_ii}
		the map
		\[
			\alpha_\rho \colon M \cap \sigma^\vee \cap \rho^\bot \to M' \cap (\sigma')^\vee \cap (\rho')^\bot \, , \quad
			m \mapsto \Upsilon(e_0+m)-\Upsilon(e_0)
		\]
		is well-defined and bijective. 
		
		\begin{claim}
			\label{claim.alpha_rho_semi-group_homo}
			$\alpha_\rho$ is a semi-group homomorphism.
		\end{claim}
		
		Indeed, $\alpha_\rho(0) = \Upsilon(e_0)-\Upsilon(e_0)  = 0$ and for 
		$p, m \in M \cap \sigma^\vee \cap \rho^\bot$ we have
		\begin{eqnarray*}
			\alpha_\rho(p) + \alpha_\rho(m) &=&
			\Upsilon(e_0 + p)-\Upsilon(e_0) + \Upsilon(e_0 + m)-\Upsilon(e_0) \\
			&=& \Upsilon(e_0 + p+m)-\Upsilon(e_0+m) + \Upsilon(e_0 + m)-\Upsilon(e_0) \\
			&=&\Upsilon(e_0 + p+m) - \Upsilon(e_0) \\
			&=& \alpha_\rho(p+m) \, ,
		\end{eqnarray*}
		where the second equality follows from 
		Proposition~\ref{prop.Theta_and_the_facets}\eqref{prop.Theta_and_the_facets_0},\eqref{prop.Theta_and_the_facets_i}.
		This implies Claim~\ref{claim.alpha_rho_semi-group_homo}.
		
		\medskip
		
		Note that the semi-group spanned by $M \cap \sigma^\vee \cap \rho^\bot$ inside $M$ is equal to
		$\rho^\bot \cap M$ and the semi-group spanned by $M' \cap (\sigma')^\vee \cap (\rho')^\bot$ 
		inside $M'$ is equal to $(\rho')^\bot \cap M'$ by Lemma~\ref{lem.span_facet}. Hence, $\alpha_\rho$ extends
		to a group isomorphism $\rho^\bot \cap M \to (\rho')^\bot \cap M'$ that itself extends
		to an $\RR$-linear isomorphism $\rho^\bot \to (\rho')^\bot$ that we denote again by 
		$\alpha_\rho$. We define $\psi_\rho \colon M_\RR \to M'_{\RR}$ by
		\[
			\psi_\rho(s e_0 + u) = s \Upsilon(e_0) + \alpha_\rho(u) \quad
			\textrm{for all $s \in \RR$ and all $u \in \rho^\bot$} \, .
		\] 
		Hence, $\psi_\rho$ is an $\RR$-linear isomorphism. 
		By Lemma~\ref{lem.Chara_all_H_e_rho_d_finite}, the complement of $B_{E, \rho, D}$
		in $S_\rho$ is finite. Let $e \in B_{E, \rho, D}$. Take any 
		$m_0 \in M$ inside the relative interior of $\sigma^\vee \cap \rho^\bot$. 
		Then, we may choose $t \in \ZZ_{\geq 0}$
		such that $\sprod{e_0 + tm_0 -e}{v_\mu} \geq 0$ for all extremal rays $\mu$ of $\sigma$. 
		Hence, $m \coloneqq e_0 + tm_0 -e \in M \cap \sigma^\vee \cap \rho^\bot$ and thus we get
		\begin{eqnarray*}
			\psi_\rho(e) &=& \psi_\rho(e_0 + tm_0)- \psi_\rho(m) \\
			&=& \Upsilon(e_0) + \alpha_\rho(tm_0)-\alpha_\rho(m) \\
			&=& \Upsilon(e_0) + \Upsilon(e_0 + tm_0)-\Upsilon(e_0)-(\Upsilon(e_0+m)-\Upsilon(e_0)) \\
			&=& \Upsilon(e+m) -\Upsilon(e_0+m)+\Upsilon(e_0) \\
			&=& \Upsilon(e+m) -\Upsilon(e+m)+\Upsilon(e)  \\
			&=&  \Upsilon(e) \, ,
		\end{eqnarray*}
		where the second last equality follows from Proposition~\ref{prop.Theta_and_the_facets}\eqref{prop.Theta_and_the_facets_i}.
		Hence, $\psi_\rho$ and $\Upsilon$ agree on almost every element from $S_\rho$.
		Now, we only have to check that $\psi_\rho$ satisfies~\eqref{cor.facets_1}-\eqref{cor.facets_3}.
		
		Statement~\eqref{cor.facets_1} follows from the following calculation:
		\begin{eqnarray*}
			\psi_\rho(M) = \psi_\rho(\ZZ e_0 \oplus (\rho^\bot \cap M))
							   &=& \psi_\rho(\ZZ e_0) \oplus \psi_\rho(\rho^\bot \cap M) \\
							   &=& \ZZ \Upsilon(e_0) \oplus ((\rho')^\bot \cap M')  = M' \, .
		\end{eqnarray*}
		By definition,~\eqref{cor.facets_2} is satisfied. By construction,
		we have that $\psi_\rho(M \cap \sigma^\vee \cap \rho^\bot) = M' \cap (\sigma')^\vee \cap (\rho')^\bot$. 
		Hence, statement~\eqref{cor.facets_3} is satisfied, as the convex cone
		generated by $M \cap \sigma^\vee \cap \rho^\bot$ inside 
		$M_\RR$ is equal to $\sigma^\vee \cap \rho^\bot$
		and as an analogous statement holds for $(\sigma')^\vee \cap (\rho')^\bot$.
	\end{proof}

	The next consequence of Proposition~\ref{prop.Theta_and_the_facets} says that
	for extremal rays $\mu_1$, $\mu_2$ of $\sigma$ the previously defined linear isomorphisms
	$\psi_{\mu_2}, \psi_{\mu_1} \colon M_{\RR} \to M'_{\RR}$ 
	coincide on $\sigma^\vee \cap \mu_1^\bot \cap \mu_2^\bot$.

	\begin{corollary}
		\label{cor.intersection_of_facets}
		Assume that $\sigma^\vee$ is strongly convex and 
		let $\mu_1, \mu_2$ be extremal rays of $\sigma$.
		Let $\psi_{\mu_1}, \psi_{\mu_2} \colon M_{\RR} \to M'_{\RR}$ be the $\RR$-linear isomorphisms
		defined in Corollary~\ref{cor.facets}. Then 
		\[
			\psi_{\mu_1}(u) = \psi_{\mu_2}(u) \quad 
			\textrm{for all $u \in \sigma^\vee \cap \mu_1^\bot \cap \mu_2^\bot$} \, .
		\]
	\end{corollary}

	\begin{proof}
		If $n = 1$, then $\sigma$ has exactly one extremal ray (as $\sigma^\vee$ is strongly convex) 
		and thus nothing is to prove. Hence, we may assume $n \geq 2$.
		
		Let $m \in M \cap \sigma^\vee \cap \mu_1^\bot \cap \mu_2^\bot \setminus \{0\}$. 
		\begin{claim}
			\label{claim.lin_dependent}
			$\psi_{\mu_1}(m)$, $\psi_{\mu_2}(m)$ are linearly dependent in $M'$.
		\end{claim}
	
		Indeed, assume towards a contradiction that $\psi_{\mu_1}(m)$, $\psi_{\mu_2}(m)$ are linearly
		independent in $M'$. Since $D(X')$ generates $M'$ as  a group
		(here we use that $(\sigma')^\vee$ is strongly convex, see~\cite[Corollary~6.10]{ReSa2021Characterizing-smo})
		we may choose
		$e_3, \ldots, e_n \in D(X)$ such that 
		$\psi_{\mu_1}(m), \psi_{\mu_2}(m), \Upsilon(e_3), \ldots, \Upsilon(e_n)$
		are linearly independent in $M'$. Moreover, choose $e_i \in S_{\mu_i}$ for $i=1,2$.
		For $i=1, 2$, using that $\psi_{\mu_i}$ coincides with $\Upsilon$ on a cofinite
		subset of $S_{\mu_i}$ and that $m \neq 0$, we get
		\[
			\psi_{\mu_i}(e_i) + t \psi_{\mu_i}(m) = \psi_{\mu_i}(e_i + t m) = \Upsilon(e_i + tm) \quad
			\textrm{for $t \in \ZZ_{\geq 0}$ large enough} \, .
		\] 
		We get now the following contradiction
		\begin{eqnarray*}
			0 &\neq& |\det(\psi_{\mu_1}(m), \psi_{\mu_2}(m), \Upsilon(e_3), \ldots, \Upsilon(e_n))| \\
				&=& 
				\lim_{t \to \infty} \frac{1}{t^2} \left|\det(\psi_{\mu_1}(e_1) + t \psi_{\mu_1}(m), 
				\psi_{\mu_2}(e_2) + t \psi_{\mu_2}(m), \Upsilon(e_3), \ldots, \Upsilon(e_n)) \right| \\
				&=&
				\lim_{t \to \infty} \frac{1}{t^2} \left|\det(\Upsilon(e_1 + t m), 
				\Upsilon(e_2 + t m), \Upsilon(e_3), \ldots, \Upsilon(e_n)) \right| \\
				&\stackrel{\textrm{\ref{property.Upsilon_2}}}{=}&
				\lim_{t \to \infty} \frac{1}{t^2} \left|\det(e_1 + t m, e_2 + t m, e_3, \ldots, e_n) \right|
				= \left|\det(m, m, e_3, \ldots, e_n) \right| \, .
		\end{eqnarray*}
		This proves Claim~\ref{claim.lin_dependent}.
		
		\medskip
		
		Let $m_0 \in M$ be the unique primitive vector such that
		there exists $s \in \ZZ_{> 0}$ with $m = sm_0$ (here we use $m \neq 0$).
		As $\psi_{\mu_i}(M) = M'$ for $i=1, 2$ (see Corollary~\ref{cor.facets}\eqref{cor.facets_1}),
		$\psi_{\mu_1}(m_0)$, $\psi_{\mu_2}(m_0)$ are linearly dependent 
		primitive vectors in $M'$ due to Claim~\ref{claim.lin_dependent}. 
		Hence,  $\psi_{\mu_1}(m_0)$ and $\psi_{\mu_2}(m_0)$ are the same up to a sign.
		As both elements belong to $(\sigma')^\vee$ and since $(\sigma')^\vee$ is strongly convex,
		we get $\psi_{\mu_1}(m_0) = \psi_{\mu_2}(m_0)$. Hence,
		\[
			\psi_{\mu_1}(m) = s \psi_{\mu_1}(m_0) = s \psi_{\mu_2}(m_0) = \psi_{\mu_2}(m) \, .
		\]
		As $m \in M \cap \sigma^\vee \cap \mu_1^\bot \cap \mu_2^\bot \setminus \{0\}$ was arbitrary and since
		$\sigma^\vee \cap \mu_1^\bot \cap \mu_2^\bot$ is the convex
		cone generated by $M \cap \sigma^\vee \cap \mu_1^\bot \cap \mu_2^\bot$
		inside $M_{\RR}$, the result follows.
	\end{proof}

	The next proposition will be the key in the proof of the main theorem.

	\begin{proposition}
		\label{prop.psi}
		Assume that $\sigma^\vee$ is strongly convex.
		Then there exists a bijection
		\[
			\psi \colon \partial \sigma^\vee \to \partial (\sigma')^\vee
		\] 
		such that
		\begin{enumerate}
			\item \label{prop.psi_1} $\psi(M \cap \partial \sigma^\vee) = M' \cap \partial (\sigma')^\vee$;
			\item \label{prop.psi_2} there exists $\varepsilon \in \{\pm1\}$ with
			$\det(\psi(u_1), \ldots, \psi(u_n)) =  \varepsilon \det(u_1, \ldots, u_n)$
			for all $u_1, \ldots, u_n \in \partial \sigma^\vee$.
		\end{enumerate}
	\end{proposition}

	\begin{proof}
		For $i=1, \ldots, r$, let $\psi_{\rho_i} \colon M_\RR \to M'_\RR$ be the $\RR$-linear isomorphism
		from Corollary~\ref{cor.facets}. Using Corollary~\ref{cor.intersection_of_facets}
		and property~\ref{property.Upsilon_1},
		there exists a bijective map $\psi \colon \partial \sigma^\vee \to \partial (\sigma')^\vee$
		such that $\psi |_{\sigma^\vee \cap \rho_i^\bot} = \psi_{\rho_i} |_{\sigma^\vee \cap \rho_i^\bot}$
		for all $i = 1, \ldots, r$.
		
		\eqref{prop.psi_1}: Since 
		$\psi_{\rho_i}(M \cap \sigma^\vee \cap \rho_i^\bot) = M' \cap (\sigma')^\vee \cap (\rho'_i)^\bot$
		for all $i=1, \ldots, r$ (see property~\ref{property.Upsilon_1} and 
		Corollary~\ref{cor.facets}\eqref{cor.facets_1}, \eqref{cor.facets_3}), the statement follows.
		
		\eqref{prop.psi_2}:
		If $n = 1$, then $\sigma$ has exactly one extremal ray $\rho$ and thus the statement 
		follows from $\psi_{\rho}(0) = 0$. 
		Hence, we assume $n \geq 2$.
		We establish first two claims.
		
			\begin{claim}
			\label{Claim3}
			Let $1 \leq i \leq n$, let 
			$j_1, \ldots, j_{i-1}, j_{i+1}, \ldots, j_n \in \{1, \ldots, r\}$ and let $u \in M_{\RR}$ be a non-zero element such that
			\begin{equation}
				\label{Eq.Generating}
				\Span_{\RR} \left( \{u\} \cup \bigcup_{k \neq i} \sigma^\vee \cap \rho_{j_k}^\bot \right) = M_{\RR} \, .
			\end{equation}
			Then, there exist $u_{k} \in \rho_{j_k}^\bot \cap \sigma^\vee$, $k = 1, \ldots, i-1, i+1, \ldots, n$ 
			such that
			\[
			\det(u_1, \ldots, u_{i-1}, u, u_{i+1}, \ldots, u_n) \neq 0 \, .
			\]
		\end{claim}
		
		Indeed, without loss of generality, we may assume $i = 1$.
		We construct inductively the elements $u_2, \ldots, u_n \in M_{\RR}$.
		For $1 \leq k < n$, assume that 
		$u_2, \ldots, u_k$ are already constructed such that $u, u_2, \ldots u_k$
		are linearly independent and 
		$u_l \in (\rho_{j_l}^\bot \cap \sigma^\vee) \setminus \bigcup_{a \neq j_l} \rho_a^\bot$
		for $l = 2, \ldots, k$. If the span of the elements $u, u_2, \ldots, u_k$ in $M_{\RR}$
		is equal to $\rho_{j_{k+1}}^\bot$, then $k = n-1$ and $j_l = j_{n}$ for $l =2, \ldots, n$. 
		However, this would then contradict the assumption~\eqref{Eq.Generating}.
		Hence, there exists an element $u_{k+1} \in \rho_{j_{k+1}}^\bot \cap \sigma^\vee$ 
		such that $u_{k+1} \not\in \Span_{\RR} \{u, u_2, \ldots, u_k\}$ and $u_{k+1} \not\in \rho_a^\bot$
		for all $a \neq j_{k+1}$. This finishes our construction of the elements $u_2, \ldots, u_{n}$
		and proves Claim~\ref{Claim3}.

		\begin{claim} 
		\label{Claim1}
		If $(i_1, \ldots, i_n) \in \{1, \ldots, r\}^n$ such that 
		$i_1, \ldots, i_n$ are not all equal, then there exists a unique
		$\varepsilon_{i_1, \ldots, i_n} \in \{\pm1\}$ with
		\[
			\det(\psi_{\rho_{i_1}}(u_1), \ldots, \psi_{\rho_{i_n}}(u_n)) = \varepsilon_{i_1, \ldots, i_n} \det(u_1, \ldots, u_n)
		\]
		for all $(u_1, \ldots, u_n) \in \rho_{i_1}^\bot \times \cdots \times \rho_{i_n}^\bot$.
		In particular, $\varepsilon_{i_1, \ldots, i_n} = \varepsilon_{\xi(i_1), \ldots, \xi(i_n)}$
		for all permutations $\xi$ of $\{1, \ldots, r\}$.
		\end{claim}
		
		Indeed, let $m_j \in \sigma^\vee \cap \rho_{i_j}^\bot \cap M$, $e_j \in S_{\rho_{i_j}}$ 
		for $j =1,\ldots, n$.
		By Corollary~\ref{cor.facets}, for $t \in \ZZ_{\geq 0}$ large enough, we get
		$\psi_{\rho_{i_j}}(e_j + tm_j) = \Upsilon(e_j + tm_j)$ for all $j=1, \ldots, n$. Hence,
		\begin{eqnarray*}
				&& |\det(\psi_{\rho_{i_1}}(m_1), \ldots, \psi_{\rho_{i_n}}(m_n))| \\
				&=& \lim_{t \to \infty} 
				\frac{1}{t^n} 
				|\det(\psi_{\rho_{i_1}}(e_1 + tm_1), \ldots, \psi_{\rho_{i_n}}(e_n + tm_n))| \\
				&=& \lim_{t \to \infty} \frac{1}{t^n}  | \det(\Upsilon(e_1 + tm_1), \ldots, \Upsilon(e_n + tm_n))| \\
				&\stackrel{\textrm{\ref{property.Upsilon_2}}}{=}& 
				\lim_{t \to \infty} \frac{1}{t^n}  | \det(e_1 + tm_1, \ldots, e_n + tm_n)| = | \det(m_1, \ldots, m_n)| \, .
		\end{eqnarray*}
		Consider the two real polynomial maps
		\begin{eqnarray*}
			\eta \colon \rho_{i_1}^\bot \times \cdots \times \rho_{i_n}^\bot \to \RR \, , &\quad&
			(u_1, \ldots, u_n) \mapsto \det(u_1, \ldots, u_n) \\
			\mu \colon \rho_{i_1}^\bot \times \cdots \times \rho_{i_n}^\bot \to \RR \, , &\quad&
			(u_1, \ldots, u_n) \mapsto \det(\psi_{\rho_{i_1}}(u_1), \ldots, \psi_{\rho_{i_n}}(u_n)) \, .
		\end{eqnarray*}
		We proved that $\eta^2$ and $\mu^2$ agree on the Zariski dense subset 
		$(\sigma^\vee \cap \rho_{i_1}^\bot \cap M) \times \cdots \times (\sigma^\vee \cap\rho_{i_n}^\bot \cap M)$.
		Hence, $\eta^2 = \mu^2$.
		Note that $\eta$ and $\mu$ are non-zero, as not all $i_1, \ldots, i_n$ are equal, see e.g.~Claim~\ref{Claim3}.
		Since $\mu^2 = \eta^2$, the prime factors of the two real polynomials 
		$\eta$ and $\mu$ are equal up to  elements from $\RR \setminus \{0\}$.
		This implies Claim~\ref{Claim1}.

		%
		
		\medskip
		
		In order to prove~\eqref{prop.psi_2}, 
		we have to show for all $i_1, \ldots, i_n, i_1', \ldots, i_n' \in \{1, \ldots, r\}$ where
		not all $i_1, \ldots, i_n$ are equal and not all $i_1', \ldots, i_n'$ are equal that 
		\begin{equation}
			\label{Eq.ToShow}
			\varepsilon_{i_1, \ldots, i_n} = \varepsilon_{i_1', \ldots, i_n'} \, .
		\end{equation}
		If $n = 2$, then $r = 2$ and~\eqref{Eq.ToShow} follows directly from Claim~\ref{Claim1}.
		Hence, we assume $n \geq 3$.
		In order to prove~\eqref{Eq.ToShow}, it is enough to consider the case where
		$i_j = i_j'$ for $j=2, \ldots, n$.
		
		Note the following: if $a, b \in \{1, \ldots, r\}$ such that $a, i_2, \ldots, i_n$
		are not all equal and $b, i_2, \ldots, i_n$ are  not all equal
		and if there exist $u \in \sigma^\vee \cap \rho_a^\bot \cap \rho_b^\bot$,
		$u_j \in \sigma^\vee \cap\rho_{i_j}^\bot$, $j=2, \ldots, n$, 
		such that $u, u_2, \ldots, u_n$ are linearly independent,
		then 
		\begin{equation}
			\label{Eq.key} 
			\varepsilon_{a, i_2, \ldots, i_n} \xlongequal{\textrm{Claim~\ref{Claim1}}}
			\frac{\det(\psi(u), \psi(u_2) \ldots, \psi(u_n))}{\det(u, u_2, \ldots, u_n)}
			\xlongequal{\textrm{Claim~\ref{Claim1}}} \varepsilon_{b, i_2, \ldots, i_n} \, .
		\end{equation}
		As not all $i_1, \ldots, i_n$ are equal and not all $i_1', i_2 \ldots, i_n$ are equal,
		there exist non-zero $x \in \sigma^\vee \cap \rho_{i_1}^\bot$ and $y \in \sigma^\vee \cap \rho_{i_1'}^\bot$
		such that 
		\[
			\Span_{\RR}\left(\{x\} \cup \bigcup_{j = 2}^n \sigma^\vee \cap \rho_{i_j}^\bot\right)
			= M_{\RR} = 
			\Span_{\RR}\left(\{y\} \cup \bigcup_{j = 2}^n \sigma^\vee \cap \rho_{i_j}^\bot\right) \, .
		\]
		By Lemma~\ref{lem.Schoenflies_application}
		there exist $k \geq 1$ and non-zero $z_0, \ldots, z_k \in \partial \sigma^\vee$ such that 
		\begin{enumerate}
			\item $x = z_0$, $y = z_k$;
			\item for all $0 \leq l < k$ there exists a facet $\sigma^\vee \cap \rho_{r_l}^\bot$
			that contains $z_l$ and $z_{l+1}$;
			\item $\Span_{\RR}(\{z_l\} \cup \bigcup_{j=2}^{n} \sigma^\vee \cap \rho_{i_j}^\bot) = M_{\RR}$
					for all $0 \leq  l \leq k$.
		\end{enumerate}
		By Claim~\ref{Claim3}, for $0 \leq l \leq k$, there exist 
		$u_{lj} \in \sigma^\vee \cap \rho_{i_j}^\bot$, $j =2, \ldots, n$ 
		such that $z_l, u_{l2}, \ldots, u_{ln}$
		are linearly independent. Since
		$z_l \in \sigma^\vee \cap \rho_{r_{r_{l-1}}}^\bot \cap \rho_{r_{l}}^\bot$
		for all $0 < l < k$ and since $z_0 \in \sigma^\vee \cap \rho_{r_{i_1}}^\bot \cap \rho_{r_{0}}^\bot$,
		$z_k \in \sigma^\vee \cap \rho_{r_{i'_1}}^\bot \cap \rho_{r_{k-1}}^\bot$, ~\eqref{Eq.key} implies 
		\[
			\varepsilon_{i_1, i_2, \ldots, i_n} = 
			\varepsilon_{r_0, i_2, \ldots, i_n} =
			\varepsilon_{r_1, i_2, \ldots, i_n} =
			\cdots = \varepsilon_{r_{k-1}, i_2, \ldots, i_n} = \varepsilon_{i_1', i_2,  \ldots, i_n} \, .
		\]
		This proves~\eqref{Eq.ToShow} and gives the second statement of the proposition.
	\end{proof}

	Using Proposition~\ref{prop.psi} we prove now that there exists an
	$\RR$-linear isomorphism $\Psi \colon M_{\RR} \to M'_{\RR}$ that sends $M$ onto $M'$
	and $\sigma^\vee$ onto $(\sigma')^\vee$. In particular this corollary says, that $X$ and $X'$
	are isomorphic as toric varieties.

	\begin{corollary}
		\label{cor.psi_linear_and_preserves_M}
		Assume that $\sigma^\vee$ is strongly convex.
		Then there exists an $\RR$-linear 
		isomorphism $\Psi \colon M_{\RR} \to M'_{\RR}$ 
		with $\Psi(M) = M'$ and $\Psi(\sigma^\vee) = \Psi((\sigma')^\vee)$.
	\end{corollary}

	\begin{proof} 
		As $\sigma^\vee$ and $(\sigma')^\vee$ are both strongly convex, the statement is trivial in case
		$n = 1$. Hence, we assume $n \geq 2$.
		Let $\psi \colon \partial \sigma^\vee \to \partial (\sigma')^\vee$
		be the  bijective map from Proposition~\ref{prop.psi}.
		
		\begin{claim}
			\label{claim.well-definedness_psi}
			If $\lambda_1, \ldots, \lambda_k \in \RR$ and $u_1, \ldots, u_k \in \partial \sigma^\vee$ such that 
			$\sum_{i=1}^k \lambda_i u_i = 0$, then $\sum_{i=1}^k \lambda_i \psi(u_i) = 0$.
		\end{claim}
	
		Indeed, assume towards a contradiction that $\sum_{i=1}^k \lambda_i \psi(u_i) \neq 0$.
		Since $\psi(\partial \sigma^\vee) = \partial (\sigma')^\vee$ spans $M'_{\RR}$
		(as $(\sigma')^\vee$ is strongly convex and $n \geq 2$), 
		there exist elements $z_1, \ldots, z_{n-1} \in \partial \sigma^\vee$ such that
		\begin{eqnarray*}
			D \coloneqq \sum_{i=1}^k \lambda_i \det(\psi(z_1), \ldots, \psi(z_{n-1}), \psi(u_i))
			= \det\left(\psi(z_1), \ldots, \psi(z_{n-1}), \sum_{i=1}^k \lambda_i \psi(u_i)\right)
		\end{eqnarray*}
		is non-zero.
		By Proposition~\ref{prop.psi}\eqref{prop.psi_2} there exists $\varepsilon \in \{\pm 1\}$ such that
		\[
			\varepsilon D = \sum_{i=1}^k \lambda_i \det(z_1, \ldots, z_{n-1}, u_i) 
			= \det\left(z_1, \ldots, z_{n-1}, \sum_{i=1}^k \lambda_i u_i \right) = 0 \, .
		\]
		This contradicts $D \neq 0$ and proves Claim~\ref{claim.well-definedness_psi}.
		
		\medskip
		
		If $u \in M_{\RR}$, take $u_1, \ldots, u_k \in \partial \sigma^\vee$ and 
		$\lambda_1, \ldots, \lambda_k \in \RR$ such that $u = \sum_{i=1}^k \lambda_i u_i$.
		By Claim~\ref{claim.well-definedness_psi}, the element 
		$\Psi(u) \coloneqq \sum_{i=1}^k \lambda_i \psi(u_i) \in M_{\RR}$ is well-defined, i.e.~independent
		of the choice of $\lambda_1, \ldots, \lambda_k \in \RR$ and $u_1, \ldots, u_k \in \partial \sigma^\vee$.
		Hence, the map
		\[
			\Psi \colon M_{\RR} \to M'_{\RR} \, , \quad u \mapsto \Psi(u)
		\]
		is well-defined. By construction, $\Psi$ is $\RR$-linear and $\Psi$ extends $\psi$. 
		Since $\sigma^\vee$, $(\sigma')^\vee$ are strongly convex and since $n \geq 2$, 
		$\sigma$ and $\sigma'$ have
		at least two distinct extremal rays. By Lemma~\ref{lem.span_facet} we get thus
		$\Span_{\ZZ}(M \cap \partial \sigma^\vee) = M$ and 
		$\Span_{\ZZ}(M' \cap \partial (\sigma')^\vee) = M'$. Proposition~\ref{prop.psi}\eqref{prop.psi_1}
		and the linearity of $\Psi$ imply thus $\Psi(M) = M'$. In particular, $\Psi$
		is an $\RR$-linear isomorphism. Moreover, the
		convex cone generated by $\partial \sigma^\vee$ ($\partial (\sigma')^\vee$)  is equal to 
		$\sigma^\vee$ ($(\sigma')^\vee$). Hence, $\Psi(\sigma^\vee) = (\sigma')^\vee$.
	\end{proof}

	\begin{proof}[Proof of the main theorem]
		By Proposition~\ref{prop.reduction_strongly_convex} we are reduced to the case
		when $\sigma^\vee$ is strongly convex. Hence, the statement follows from Corollary~\ref{cor.psi_linear_and_preserves_M}.
	\end{proof}

	\subsection{The case when a bijection $D(X) \to D(X')$ extends to a group isomorphism $M \to M'$}
	We finish this section by showing that a bijection $D(X) \to D(X')$ which 
	extends to a group isomorphism $M \to M'$ automatically satisfies the properties~\ref{thm.properties_bijection_Theta_1},~\ref{thm.properties_bijection_Theta_2}
	and~\ref{thm.properties_bijection_Theta_3} 
	of the main theorem. 

	\begin{proposition}
		\label{Prop.Reconstruction}
		If $\phi \colon D(X) \to D(X')$ is a bijection
		that extends to a group-iso\-mor\-phism $\Phi \colon M \to M'$, then $\Phi$ maps
		$\Lambda(X)$ onto $\Lambda(X')$ and hence $\phi$ satisfies the 
		properties~\ref{thm.properties_bijection_Theta_1},~\ref{thm.properties_bijection_Theta_2}
		and~\ref{thm.properties_bijection_Theta_3} 
		of the main theorem. 
	\end{proposition}

	\begin{proof}
		The proof is a direct consequence of the following statement, which can be found
		in~\cite[Corollary~8.4]{ReSa2021Characterizing-smo}: if $Y$ is an affine 
		$T$-toric variety, then exactly one of the following cases occur
		(the asymptotic cones and convex hulls are all taken inside $\frak{X}(T)_{\RR}$):
		\begin{itemize}
			\item
				$\dim \Conv(D(Y)_{\infty}) = n$, $D(Y) \neq \varnothing$ and
				$\Lambda(Y) = \Conv ( D(Y)_{\infty}) \cap \frak{X}(T)$;
		\item
				$\dim \Conv(D(Y)_{\infty}) < n$, $D(Y)\neq \varnothing$,
				$D(Y)_{\infty}$ is a hyperplane in $\frak{X}(T)_{\RR}$ and
				$\Lambda(Y) = H^+ \cap \frak{X}(T)$,
				where $H^+ \subset \frak{X}(T)_{\RR}$
				is the closed half space with boundary $D(Y)_{\infty}$ 
				that does not intersect $D(Y)$;
		\item
			 $D(Y) = \varnothing$ and $\Lambda(Y) = \frak{X}(T)$. \qedhere
		\end{itemize}
	\end{proof}
	
	\providecommand{\bysame}{\leavevmode\hbox to3em{\hrulefill}\thinspace}
	\providecommand{\MR}{\relax\ifhmode\unskip\space\fi MR }
	\providecommand{\MRhref}[2]{%
	  \href{http://www.ams.org/mathscinet-getitem?mr=#1}{#2}
	}
	\providecommand{\href}[2]{#2}

\end{document}